\documentclass[11pt, reqno]{amsart}
\usepackage{amssymb, amsthm, amsmath, amsfonts,mathtools}
\usepackage{array, epsfig}
\usepackage{bbm}

\usepackage{hyperref}
\usepackage{url}
\usepackage[numbers,square]{natbib}
\usepackage{color}
\allowdisplaybreaks

\numberwithin{equation}{section}
  \evensidemargin
\oddsidemargin
\newcommand{\stirling}[2]{\genfrac{[}{]}{0pt}{}{#1}{#2}}
\newcommand{\stirlingsec}[2]{\genfrac{\{}{\}}{0pt}{}{#1}{#2}}
\newcommand{\stirlingb}[2]{B\hspace*{-0.2mm}\genfrac{[}{]}{0pt}{}{#1}{#2}}
\newcommand{\stirlingsecb}[2]{B\hspace*{-0.2mm}\genfrac{\{}{\}}{0pt}{}{#1}{#2}}

\newcommand{\eulerian}[2]{\genfrac{\langle}{\rangle}{0pt}{}{#1}{#2}}
\newcommand{\eulerianb}[2]{B\hspace*{-0.2mm}\genfrac{\langle}{\rangle}{0pt}{}{#1}{#2}}

\newcommand{\ii}{{\rm{i}}}

\newcommand{\bB}{\mathbb{B}}

\newcommand{\bS}{\mathbb{S}}

\newcommand{\bJ}{\mathbb{J}}

\newcommand{\cC}{\mathcal{C}}
\newcommand{\cD}{\mathcal{D}}

\newcommand{\cQ}{\mathcal{Q}}
\newcommand{\cP}{\mathcal{P}}

\newcommand{\Bin}{\text{\rm Bin}}

\DeclareMathOperator{\relint}{relint}

\def\dint{\textup{d}}

\newcommand{\E}{\mathbb E}

\newcommand{\R}{\mathbb{R}}
\newcommand{\N}{\mathbb{N}}

\newcommand{\C}{\mathbb{C}}
\newcommand{\Z}{\mathbb{Z}}

\renewcommand{\P}{\mathbb{P}}

\renewcommand{\Re}{\operatorname{Re}}
\renewcommand{\Im}{\operatorname{Im}}

\newcommand{\conv}{\mathop{\mathrm{conv}}\nolimits}
\newcommand{\pos}{\mathop{\mathrm{pos}}\nolimits}
\newcommand{\lin}{\mathop{\mathrm{lin}}\nolimits}

\newcommand{\eps}{\varepsilon}
\newcommand{\eqdistr}{\stackrel{d}{=}}

\newcommand{\bsl}{\backslash}
\newcommand{\ind}{\mathbbm{1}}

\newcommand{\dd}{{\rm d}}
\newcommand{\eee}{{\rm e}}

\theoremstyle{plain}
\newtheorem{theorem}{Theorem}[section]
\newtheorem{lemma}[theorem]{Lemma}
\newtheorem{corollary}[theorem]{Corollary}

\theoremstyle{definition}

\newtheorem{example}[theorem]{Example}

\theoremstyle{remark}
\newtheorem{remark}[theorem]{Remark}

\makeatletter
\@namedef{subjclassname@2020}{\textup{2020} Mathematics Subject Classification}
\makeatother

\begin{document}

\author{Zakhar Kabluchko}
\address{Zakhar Kabluchko: Institut f\"ur Mathematische Stochastik,
Universit\"at M\"unster,
Orl\'eans-Ring 10,
48149 M\"unster, Germany}
\email{zakhar.kabluchko@uni-muenster.de}

\author{Hugo Panzo}
\address{Hugo Panzo: Department of Mathematics and Statistics, Saint Louis University, St. Louis, USA}
\email{hugo.panzo@slu.edu}

\title[A refinement of the Sylvester problem]{A refinement of the Sylvester problem: Probabilities of combinatorial types}



\begin{abstract}
Let $X_1,\ldots, X_{d+2}$ be random points in $\R^d$. The classical Sylvester problem asks to determine the probability that the convex hull of these points, denoted by $P:= [X_1,\ldots, X_{d+2}]$, is a simplex. In the present paper, we study a refined version of this problem which asks to determine the probability that $P$ has a given combinatorial type. It is known that there are $\lfloor d/2\rfloor+1$ possible combinatorial types of simplicial $d$-dimensional polytopes with at most $d+2$ vertices. These types are denoted by $T_0^d, T_1^d, \ldots, T_{\lfloor d/2 \rfloor}^d$, where $T_0^d$ is a simplex with $d+1$ vertices, while the remaining types have exactly $d+2$ vertices. Our aim is thus to compute the probability
$$
p_{d,m} := \P[P \text{ is of type } T_{m}^d], \qquad m\in \{0,1,\ldots, \lfloor d/2 \rfloor\}.
$$
The classical Sylvester problem corresponds to the case $m=0$. We shall compute $p_{d,m}$ for all $m$ in the following cases: (a) $X_1,\ldots, X_{d+2}$ are i.i.d.\ normal; (b)  $X_1,\ldots, X_{d+2}$ follow a $d$-dimensional beta or beta prime distribution, which includes the uniform distribution on the ball or on the sphere as special cases; (c) $X_1,\ldots, X_{d+2}$ form a random walk with  exchangeable increments. As a by-product of case (a) we  recover a recent solution to Youden's demon problem which asks to determine the probability that, in a one-dimensional  i.i.d.\ normal sample $\xi_1,\ldots, \xi_n$, the empirical mean $\frac 1n (\xi_1 + \ldots + \xi_n)$ lies between the $k$-th and the $(k+1)$-st order statistics. We also consider the conic (or spherical) version of the refined Sylvester problem and solve it in several special cases.
\end{abstract}

\maketitle

\section{Introduction}
\subsection{Classical Sylvester's problem and its refinement}
Consider $n\geq d+1$ random points $X_1,\ldots, X_{n}$ in $\R^d$ with some specified joint distribution.   The classical Sylvester problem asks to determine the probability that these points are in convex position meaning that their convex hull, denoted by $P:= [X_1,\ldots, X_{n}]$, is a polytope with $n$ vertices. This problem has been extensively  studied for points that are i.i.d.\ and  uniformly distributed in a convex set; see for example~\cite{BaranyBullSurvey,CalkaSurvey,hug_rev,morin_n_points_convex_position_regulag_k_gon,reitzner_random_polys_survey},  \cite[pp.\ 63--65]{lS76}, \cite{schneider_discrete_aspects_stoch_geom}, \cite[Chapter~8]{schneider_weil_book}, \cite[Chapter 5]{Solomon1975}.

To exclude degenerate cases, we assume that the points $X_1,\ldots, X_n$ are in general  affine position a.s.\ meaning that no $d+1$ points  from this collection lie on a common affine hyperplane. Since for $n=d+1$ points, $P$ is a.s.\ a $d$-dimensional simplex, the simplest non-trivial case of the Sylvester problem is when $n=d+2$. For example, for i.i.d.\ points uniformly distributed in the unit ball of $\R^d$ (or, more generally, in any $d$-dimensional ellipsoid), Kingman~\cite{kingman_secants} proved that
\begin{equation}\label{eq:kingman_sylvester_simplex}
\P[X_1,\ldots, X_{d+2} \text{ are in convex position}] = 1 -  \frac{d+2}{2^d} \cdot  \binom{d+1}{(d+1)/2}^{d+1} \cdot  \binom{(d+1)^2}{(d+1)^2/2}^{-1}.
\end{equation}
In general, for $n=d+2$ points there are two possibilities (excluding events of zero probability):
\begin{itemize}
\item [(a)] either the points $X_1,\ldots, X_{d+2}$ are in convex position
\item [(b)] or one of the points falls into the convex hull of the remaining $d+1$ points (which is a $d$-dimensional simplex).
\end{itemize}
In case (a), $P$ is a $d$-dimensional polytope with $d+2$ vertices. Moreover, by the assumption of general affine position, $P$ is a \emph{simplicial} polytope (meaning that all of its facets and hence all faces are simplices).
In case (b), $P$ is a simplex with $d+1$ vertices. The Sylvester problem asks to determine the probabilities of these two cases.
While all simplices with $d+1$ vertices are combinatorially equivalent, it is known~\cite[Section~6.1]{gruenbaum_book} that there exist $\lfloor d/2 \rfloor$ different combinatorial types $T_1^d, \ldots, T_{\lfloor d/2 \rfloor}^d$ of simplicial $d$-dimensional polytopes with $d+2$ vertices. More details on these combinatorial types will be given below; for now let us only mention that type $T_m^d$ is characterized by the property that it has $(m+1)(d+1-m)$ facets.
For example, for $d=3$ there exists just one combinatorial type of simplicial polytopes with $5$ vertices  --- the union of two tetrahedra sharing a common facet, but already for $d=4$ there exist two different combinatorial types of simplicial polytopes  with $6$ vertices. It will be convenient to introduce also type $T_0^d$ by defining it to be a $d$-dimensional simplex.
It is therefore natural to subdivide  Case~(b) into $\lfloor d/2\rfloor$ subcases, which leads to the following
\begin{quote}
\textbf{Refined Sylvester problem}: For every $m\in \{0,1,\ldots, \lfloor d/2 \rfloor\}$ determine the probability that $[X_1,\ldots, X_{d+2}]$ has combinatorial type $T_m^d$.
\end{quote}



%


\subsection{Summary of results}
We shall explicitly solve the refined Sylvester problem in the following settings:
\begin{itemize}
\item $X_1,\ldots, X_{d+2}$ are i.i.d.\ standard Gaussian in $\R^d$; see Section~\ref{subsec:sylvester_refined_gaussian}.
\item $X_1,\ldots, X_{d+2}$ are i.i.d.\ with density proportional to $(1-\|x\|^2)^\beta$ on the unit ball in $\R^d$, where $\beta >-1$ is a parameter; see Section~\ref{subsec:beta_type_sylvester_refined}.
\item $X_1,\ldots, X_{d+2}$ are i.i.d.\ uniformly distributed on the unit sphere in $\R^d$; see Section~\ref{subsec:beta_type_sylvester_refined}.
\item $X_1,\ldots, X_{d+2}$ are i.i.d.\ with density proportional to $(1+\|x\|^2)^{-\beta}$ on  $\R^d$, where $\beta > \frac{d}{2} + \frac 1 {2(d+2)}$ is a parameter; see Section~\ref{subsec:beta_type_sylvester_refined}.
\item $X_1=0$, $X_1 = \xi_1 + \xi_2, \ldots, X_{d+2}= \xi_1+\ldots+\xi_{d+2}$ are points visited by a $d$-dimensional random walk with exchangeable increments $\xi_1,\ldots, \xi_{d+2}$; see Section~\ref{subsec:convex_hulls_RW_sylvester_refined}.
\end{itemize}

A recent result of~\citet{frick_newman_pegden} establishes the equivalence between the refined Sylvester problem for $n=d+2$ i.i.d.\ Gaussian points in $\R^d$ and the following Youden's problem: determine the probability that, in an i.i.d.\ standard normal sample of size $n$, the empirical mean is located between the $i$-th and the $(i+1)$-st order statistics. We shall use this equivalence to recover Kuchelmeister's recent solution~\cite{kuchelmeister_youdens_demon} to Youden's problem; see Section~\ref{subsec:youden_problem}.

The problem of determining the probabilities of combinatorial types of random polytopes of the form
$[X_1,\ldots,X_n]\subseteq\R^d$, for arbitrary $n\ge d+1$, has been considered in various sources; see, e.g.,  \cite{BokowskiRichterGebertSchindler1992OrderTypes,HenkRichterGebertZiegler1997BasicProperties,vershik_sporyshev_asymptotic_faces_random_polyhedra1992}.  In particular, it has been  conjectured~\cite{HenkRichterGebertZiegler1997BasicProperties} that the most probable combinatorial type is the cyclic polytope.  In the present paper, we focus on the case $n=d+2$, for which a simple description of the possible combinatorial types is available.


We shall also consider a conical or spherical analogue of the Sylvester problem in which $X_1,\ldots, X_{d+2}$ are random vectors in $\R^{d+1}$ in general linear position. Their positive hull, denoted by
$$
C = \pos(X_1,\ldots,X_n) := \left\{\lambda_1 X_1 + \ldots + \lambda_n X_n: \lambda_1,\ldots,\lambda_n\geq 0\right\}
$$
is a polyhedral cone in $\R^d$. Its intersection with the unit sphere, denoted by $Q:=C\cap \bS^d$, is a spherical polytope with at most $d+2$ vertices. The possible combinatorial types are now $T_{m}^d$ with $m\in\{-1,0,\ldots, \lfloor d/2\rfloor\}$, where type $T_{-1}^d$ is the whole sphere, type $T_0^d$ is a spherical simplex, and the remaining types are $d$-dimensional spherical polytopes with $d+2$ vertices.
The refined spherical Sylvester problem asks to determine the probabilities of these types. We shall solve it in the following cases:
\begin{itemize}
\item $X_1,\ldots, X_{d+2}$ random points in $\R^{d+1}$ are such that $(\pm X_1,\ldots, \pm X_{d+2})$ has the same joint distribution as $(X_1,\ldots, X_{d+2})$ for every choice of signs; see Section~\ref{subsec:wendel_donoho_tanner}.  (Example: $X_1,\ldots, X_{d+2}$ are i.i.d.\ with the uniform distribution on the $d$-dimensional unit sphere).
\item $X_1,\ldots, X_{d+2}$ are i.i.d.\ uniform on the upper half-sphere in $\R^{d+1}$; see Example~\ref{ex:refined_sylvester_cauchy_distr}.
\item $X_1,\ldots, X_{d+2}$ are points visited by a random walk or bridge in $\R^{d+1}$; see Section~\ref{subsec:walk_bridges_positive_hull_refined_sylvester}.
\end{itemize}

\section{Preliminaries}
\subsection{Simplicial \texorpdfstring{$d$}{d}-polytopes with \texorpdfstring{$d+2$}{d+2} vertices}
Following the book of~\citet[Section~6.1]{gruenbaum_book}, we recall some facts about simplicial polytopes in $\R^d$ with $d+2$ vertices. (Another useful reference is~\cite{ewald_book_combinatorial_convexity_alg_geom}).
Consider a set $I$ of $d+2$ points $x_1,\ldots, x_{d+2}$ in $\R^d$ that are in general affine position meaning that every affine hyperplane contains at most $d$ points from this set. The convex hull of these points, denoted by $P:= [x_1,\ldots, x_{d+2}]$, is either a simplex with $d+1$ vertices or a simplicial polytope with $d+2$ vertices. 

\begin{theorem}[See Section~6.1 in~\cite{gruenbaum_book}]\label{theo:comb_types_polys_d+2_vert}
There exist exactly $\lfloor d/2 \rfloor+1$ different combinatorial types of $d$-dimensional simplicial polytopes with at most $d+2$ vertices, denoted by  $T_m^d$ with $m\in \{0, 1,\ldots, \lfloor d/2 \rfloor\}$. By definition, $T_0^d$ is a simplex with $d+1$ vertices, while the types with $m \geq 1$ have exactly $d+2$ vertices. A $d$-dimensional polytope $T$ is of $m$-th type if and only if it can be represented as $T= \conv (V\cup W)$, where  $V = \{v_0,\ldots, v_m\}$ and $W= \{w_0,\ldots, w_{d-m}\}$ are disjoint sets such that the simplices $\conv V$ and $\conv W$ intersect at a single point  which belongs to the relative interior of each of the simplices. If $m\geq 1$, then the set of vertices of $T$ is $V\cup W$. If $m=0$, then $T$ is a $d$-dimensional simplex with vertices $\{w_0,\ldots, w_d\}$,  $V=\{v_0\}$ and  $v_0 \in \relint T$.
\end{theorem}

\begin{example}
For $d=3$, every simplicial polytope $T$ with $5$ vertices  can be represented as the union of two tetrahedra $[v_0, w_0, w_1,w_2]$ and $[v_1,w_0,w_1,w_2]$ sharing a common facet $[w_0,w_1,w_2]$. Its type is $T_1^3$ and $V= \{v_0,v_1\}$, $W= \{w_0,w_1,w_2\}$.
\end{example}

The above classification is closely related to a classical theorem of Radon which states that for every set $I= \{x_1,\ldots, x_{d+2}\}$ of $d+2$ points  in $\R^d$ it is possible to find a disjoint decomposition $I= V\cup W$ into non-empty subsets $V$ and $W$ such that $\conv V \cap \conv W \neq \varnothing$; see~\cite[p.~17]{barvinok_book} or~\cite[Theorem~2.3.6]{gruenbaum_book}. Such decomposition is called a Radon partition of $I$; see~\cite{eckhoff_radon_revisited} and~\cite{eckhoff_helly_radon_caratheodory} for reviews on this topic. If we additionally assume that the points in $I$ are in general affine position, the Radon partition is unique (up to swapping the sets) and $\conv V \cap \conv W = \{O\}$ consists of a single point $O$  such that  $O \in \relint \conv V$ and $O\in \relint \conv W$. If we write $\#V = m+1$ for the cardinality of $V$, then, $\# W = d+1-m$ is the cardinality of $W$. Without loss of generality, let $\# V \leq \#W$, which restricts the set of possible values of $m$ to $\{0,1,\ldots, \lfloor d/2 \rfloor\}$. As it turns out, the number $m\in \{0, 1,\ldots, \lfloor d/2 \rfloor\}$ determines the combinatorial type of the polytope $[x_1,\ldots, x_{d+2}]$. In fact, knowing the decomposition $I= V\cup W$ it is possible to determine the faces of $I$ or, more precisely, those sets $A\subseteq I$ for which $\conv A$ is a face of $P$. This is stated in the next theorem.

\begin{theorem}[See Section~6.1 in~\cite{gruenbaum_book}]\label{theo:facets_polys_d+2_vert}
Let $T = \conv (V\cup W)$ be a polytope of combinatorial type $T_{m}^d$, as in Theorem~\ref{theo:comb_types_polys_d+2_vert},  with $m \in \{0,1,\ldots,\lfloor d/2 \rfloor\}$.
\begin{itemize}
\item[(a)] Each facet of $T$ has the form $F = \conv((V\bsl\{v_i\}) \cup (W\bsl \{w_j\}))$ for some vertices $v_i\in V$ and $w_j\in W$.  Conversely, every set $F$ of the this form is a facet of $T$. Consequently, the number of facets of $T$ is
$$
f_{d-1}(T) = \# V \cdot \# W = (m+1) (d-m+1).
$$
\item[(b)] The collection of faces of $T$ (excluding $T$ istelf) coincides with the collection of all simplices  $\conv U$, where $U\subseteq V\cup W$ is such that neither $V$ nor $W$ is a subset of $U$. Consequently, for every $k\in \{0,\ldots, d-1\}$, the number of $k$-dimensional faces of $T$ is
$$
f_k (T) = \binom{d+2}{k+1} - \binom{d-m+1}{d-k+1} - \binom{m+1}{d-k+1}.
$$
Following the usual convention, $\binom nm := 0$ whenever $m\in \Z \backslash\{0,\ldots, n\}$. 
\end{itemize}
\end{theorem}

Dropping the assumption $\# V \leq \#W$ it is possible to introduce types $\tilde T_0^d, \tilde T_1^d,\ldots, \tilde T_d^d$ by defining $T$ to be of type $\tilde T_m^d$ if $\# V = m+1$ for all $m \in \{0, \ldots, d\}$.  Since the roles of $V$ and $W$ are now exchangeable, the combinatorial types $\tilde T_{m}$ and $\tilde T_{d-m}$ are in fact equal. This means that each type $T_{m}^d$ with $m \in \{0,\ldots, \lfloor d/2\rfloor\}$ is a ``union'' of two types $\tilde T_m^d$ and $\tilde T_{d-m}^d$, with one exception: if $d=2m$, then it is a union of just one type.   This explains why the following function will frequently appear in the sequel:
\begin{equation}\label{eq:sylvester_types_eta_d_m_def}
\eta_{d,m}
:=
1 + \ind_{\{m\neq d/2\}}
=
\begin{cases}
2,
&\text{ if } m\neq d/2,
\\
1, &\text{ if } m=d/2.
\end{cases}
\end{equation}

\subsection{Cones and their angles}\label{subsec:cones_angles_defs}
The positive hull of $n$ vectors $x_1,\ldots, x_n\in \R^d$ is denoted by
$$
\pos(x_1,\ldots,x_n) := \left\{\lambda_1 x_1 + \ldots + \lambda_n x_n: \lambda_1,\ldots,\lambda_n\geq 0\right\}
$$
A polyhedral cone is a positive hull of finitely many vectors in $\R^n$. Equivalently, a polyhedral cone is an intersection of finitely many half-spaces of the form $\{x\in \R^d: \langle x, u \rangle \leq 0\}$, where $u\in \R^{d}\bsl \{0\}$. Let $C$ be a polyhedral cone, $\lin C$ the minimal linear subspace containing $C$, and $\lambda_{\lin C}$ the Lebesgue measure on $\lin C$.  The angle of $C$ is defined as
$$
\alpha(C) = \frac{\lambda_{\lin C}(\bB^d \cap C)}{\lambda_{\lin C} (\bB^d \cap \lin C)}.
$$
That is, the angle is the proportion of the unit ball occupied by the cone,  in $\lin C$ as the ambient space. Note that $\alpha(C)$ takes values between $0$ and $1$.

The tangent cone of a polytope $P= [x_1,\ldots, x_n]$ at its face $F$ is defined as $T(F,P) := \pos (x_1-f,\ldots, x_n-f)$, where $f$ is any point in the relative interior of $F$. The internal angle of $P$ at $F$ is defined as $\alpha(T(F,P))$.

\section{Probabilities of types for convex hulls}

\subsection{Probabilities of types in terms of the expected \texorpdfstring{$f$}{f}-vector}
Let $X_1,\ldots, X_{d+2}$ be (possibly dependent) random points in $\R^d$. Their joint distribution is required to satisfy only the following condition: the points $X_1,\ldots, X_{d+2}$ are in general affine position with probability $1$. Then, their convex hull $P:= [X_1,\ldots, X_{d+2}]$ is a.s.\ a $d$-dimensional simplicial polytope. We are interested in computing the probabilities
$$
p_{d,m}: = \P[P \text{ is of type } T_m^{d}],
\qquad
m\in \{0,1,\ldots, \lfloor d/2 \rfloor\}.
$$
The next theorem expresses $p_{d,m}$ in terms of the expected face numbers $\E f_0(P), \ldots, \E f_m(P)$. We recall that $f_j(P)$ denotes the number of $j$-dimensional faces of $P$.
\begin{theorem}[Probabilities of types and expected $f$-vector]\label{theo:sylvester_probab_types_trhough_expected_f_vectors}
Under the above assumptions,  for all $m\in \{0,1,\ldots, \lfloor d/2 \rfloor\}$ we have
$$
p_{d,m}: = \P[P \text{ is of type } T_m^{d}] = \frac 12 \eta_{d,m} \cdot \sum_{\ell=0}^{m} (-1)^{m+\ell} \binom{d-\ell+1}{d-m+1} \left(\binom{d+2}{\ell+1} - \E f_{\ell }(P)\right),
$$
where $\eta_{d,m}$ is as in~\eqref{eq:sylvester_types_eta_d_m_def}.
\end{theorem}
\begin{proof}
Let us write  $p_m := p_{d,m}$ for the probability that $P$ is of combinatorial type $T_m^d$.  We can express $\E f_j (P)$ using the formula of the total probability as follows:
\begin{equation}\label{eq:probab_of_types_linear_eqs}
\E f_j (P) = \sum_{m = 0}^{\lfloor d/2 \rfloor} p_{m} f_j (T_m^d), \qquad j\in \{0, \ldots, d-1\}.
\end{equation}
The face numbers of combinatorial type $T_m^d$ are known; see Theorem~\ref{theo:facets_polys_d+2_vert}. It will be convenient to rewrite the formula for $f_j (T_m^d)$ as follows:
\begin{equation}\label{eq:polytope_d+2_vert_f_vector}
\binom{d+2}{j+1} - f_j (T_m^d) = \binom{d-m+1}{d-j+1} + \binom{m+1}{d-j+1},
\qquad
j\in \{0,\ldots, d-1\}.
\end{equation}

Thus, \eqref{eq:probab_of_types_linear_eqs} is a system of $d$ linear equations in the $\lfloor d/2\rfloor +1$ unknowns $p_0, \ldots, p_{\lfloor d/2 \rfloor}$. To make the number of equations equal the number of unknowns, we will consider $j\in \{0,1,\ldots, \lfloor d/2 \rfloor\}$  only. (It can be shown that the discarded equations do not contain any new information).
Using $p_{0} + \ldots + p_{\lfloor d/2 \rfloor} = 1$, let us rewrite~\eqref{eq:probab_of_types_linear_eqs} in the form
\begin{align}
y_j
:&= \binom {d+2}{j+1} -  \E f_j(P)
=
\sum_{m= 0}^{\lfloor d/2 \rfloor} \left(\binom {d+2}{j+1} - f_j(T_m^d)\right) p_{m}\notag
\\
&= \sum_{m= 0}^{j} \left(\binom{d-m+1}{j-m} + \ind_{\{m = j = d/2\}}\right) p_m,
\qquad
j\in \{0,1,\ldots, \lfloor d/2 \rfloor\}. \label{eq:probab_of_types_linear_eqs_complements}
\end{align}
Here we used that for $j\in \{0,1,\ldots, \lfloor d/2 \rfloor\}$, the formula involving $f_j(T_m^d)$ given in~\eqref{eq:polytope_d+2_vert_f_vector} simplifies as follows:
$$
\binom{d+2}{j+1} - f_j (T_m^d) = \binom{d-m+1}{d-j+1} + \ind_{\{m = j = d/2\}} = \binom{d-m+1}{j-m} + \ind_{\{m = j = d/2\}}.
$$
Consider the vectors $y = (y_0, y_1, \ldots, y_{\lfloor d/2 \rfloor})^\top$ and $p = (p_0, p_1,\ldots, p_{\lfloor d/2 \rfloor})^\top$ and the lower-triangular matrix
$$
A := (a_{jm})_{j, m = 0, \ldots, \lfloor d/2 \rfloor} \text{ with }  a_{j m} = \binom{d-m+1}{j-m} + \ind_{\{j = m = d/2\}}.
$$
The system~\eqref{eq:probab_of_types_linear_eqs_complements} takes the form $y = A p$, and its solution is $p = A^{-1} y$. Note that the matrix $A$ is lower-triangular, that is $a_{jm} = 0$ for $m>j$. Its diagonal entries are given by $a_{jj} = 1 + \ind_{\{j = m = d/2\}}$ (that is, all diagonal elements are $1$ except the last one which equals $1$ if $d$ is odd and $2$ if $d$ is even). Hence, the inverse matrix $A^{-1}$ indeed exists. We claim that it  has the form
$$
B = (b_{m\ell})_{m, \ell = 0, \ldots, \lfloor d/2 \rfloor} \text{ with }  b_{m \ell} = (-1)^{m+\ell} \binom{d-\ell+1}{m-\ell} \cdot
\frac 12 \eta_{d,m}.
$$
The matrix $B$ is also lower-triangular, that is $b_{m\ell} = 0$ for $\ell>m$. Its diagonal entries are equal to $1$ except the last one.
The role of the factor $\frac 12 \eta_{d,m}$ in the definition of $B$ is to multiply the last row of $B$ by $1/2$ if $d$ is even. Showing that $AB$ is the identity matrix amounts to proving the identity $\sum_{m=0}^{j} a_{jm} b_{m\ell} = \delta_{j\ell}$ for all $j,\ell \in \{0,\ldots, \lfloor d/2 \rfloor\}$. Note that if $d$ is even and $j= d/2$, then $a_{\lfloor d/2\rfloor \lfloor d/2\rfloor}= 2$ is multiplied with the elements of the last row of $B$, which contain an additional factor of $1/2$. Thus, in terms of binomial coefficients, the identity takes the form
\begin{equation}\label{eq:sylvester_types_faces_identity_binom_coeff}
\sum_{m=\ell}^j \binom{d-m+1}{j-m} (-1)^{m+\ell} \binom{d-\ell+1}{m-\ell}  = \delta_{j\ell}.
\end{equation}
To prove this formula, which holds for all integer $\ell \leq j \leq d+1$, we need the identity
$$
\sum_{i=0}^c (-1)^i \binom{a-i}{c-i}  \binom{a}{i} = \binom ac  \sum_{i=0}^c (-1)^i \binom ci = \ind_{\{c=0\}}, \qquad 0\leq c \leq a, \quad a,c\in \Z;
$$
see also~\cite[p.~149]{gruenbaum_book} for a more general identity. Taking $a= d - \ell + 1$, $c= j-\ell$ and changing the index of summation to $m = i + \ell$ gives~\eqref{eq:sylvester_types_faces_identity_binom_coeff}.
\end{proof}

\subsection{Gaussian random points}\label{subsec:sylvester_refined_gaussian}
The next theorem solves the refined Sylvester problem for $d+2$ i.i.d.\ Gaussian points in $\R^d$. To state it, let $\Phi(z)$ be the standard normal distribution function which is extended to an analytic function on the entire complex plane via
\begin{equation}\label{eq:Phi_def}
\Phi(z)
=
\frac 12 + \frac 1 {\sqrt{2\pi}} \int_{0}^z \eee^{-t^2/2} \dd t
=
\frac 12  + \frac 1 {\sqrt{2\pi}} \sum_{n =0}^{\infty}  \frac{(-1)^n }{(2n+1) 2^n n!} z^{2n+1}, \qquad z\in \C.
\end{equation}
The intergal is taken over any contour connecting $0$ to $z$. We shall frequently use the identity $\Phi(z) + \Phi(-z) = 1$, for all $z\in \C$. For real $z$, it follows from the symmetry of the standard Gaussian distribution, while for complex $z$ it holds by analytic continuation.
\begin{theorem}[Probabilities of types for Gaussian points]\label{theo:sylvester_probab_types_gauss}
Let $X_1,\ldots, X_{d+2}$ be independent and standard Gaussian random points in $\R^d$. Consider their convex hull $\cP_{d+2,d}=[X_1,\ldots, X_{d+2}]$. Then, for all $m\in \{0,1,\ldots, \lfloor d/2 \rfloor\}$,
\begin{align}
p_{d,m}^{\text{Gauss}}
:&=
\P[\cP_{d+2,d} \text{ is of type } T_m^{d}] \notag
\\
&=
\frac{\eta_{d,m}}{\sqrt{2\pi}} \binom{d+2}{m+1} \int_{-\infty}^{+\infty} \eee^{-x^2/2} \Phi^{m+1}\left(\frac{\ii x}{\sqrt{d+2}}\right) \Phi^{d+1-m}\left(-\frac{\ii x}{\sqrt{d+2}}\right) \dint x.
\label{eq:sylvester_refined_gaussian_formula}
\end{align}
\end{theorem}

Here, $\ii = \sqrt{-1}$. Note that the formula involves $\Phi$ of a complex argument. In the proof of this theorem and at several other occasions we shall need the following simple
\begin{lemma}\label{lem:probab_types_proof_identity_binomial_coeff}
For all  $d\in \N_0:= \{0,1,\ldots\}$ and $m\in\{0,\ldots, d+1\}$ we have
$$
\sum_{\ell=0}^m (-1)^{m+\ell}\binom{d+1-\ell}{d+1-m} \binom{d+2}{\ell+1} z^{d+1-\ell} = \binom{d+2}{m+1} \left(z^{d+1-m} (1-z)^{m+1} + (-1)^m z^{d+2}\right).
$$
\end{lemma}
\begin{proof}
Observe that
$$
\binom{d+1-\ell}{d+1-m} \binom{d+2}{\ell+1} = \frac{(d+2)!}{(d+1-m)! (m-\ell)! (\ell+1)!} = \binom{d+2}{m+1} \binom{m+1}{\ell+1}.
$$
Using this identity and then the index shift $K = \ell+1$ gives
\begin{align*}
\sum_{\ell=0}^m (-1)^{m+\ell} \binom{d+1-\ell}{d+1-m} \binom{d+2}{\ell+1} z^{d+1-\ell}
&=
\binom{d+2}{m+1} z^{d+1-m}  \sum_{\ell=0}^m \binom{m+1}{\ell+1} (-z)^{m-\ell}
\\
&=
\binom{d+2}{m+1} z^{d+1-m}  \sum_{K=1}^{m+1} \binom{m+1}{K} (-z)^{m+1-K}
\\
&=
\binom{d+2}{m+1} z^{d+1-m} \left( (1-z)^{m+1} - (-z)^{m+1}\right),
\end{align*}
and the proof of the lemma is complete.
\end{proof}
Another fact needed in the proof of Theorem~\ref{theo:sylvester_probab_types_gauss} is the following identity which can be found in~\cite[Remark~1.4]{kabluchko_zaporozhets_absorption}:
\begin{equation}\label{eq:int_gauss_density_gauss_distr_funct_complex_arg}
\int_{-\infty}^{+\infty} \eee^{-x^2/2} \Phi^{n}\left(\frac{\ii x}{\sqrt{n}}\right) \dint x = \int_{-\infty}^{+\infty} \eee^{-x^2/2} \Phi^{n}\left(-\frac{\ii x}{\sqrt{n}}\right) \dint x = 0,\qquad n \in \{2,3,\ldots\}.
\end{equation}
Note that as soon as one integral is shown to equal $0$, then the other identity follows from substituting $-x$ for $x$. We shall give an independent proof of~\eqref{eq:int_gauss_density_gauss_distr_funct_complex_arg} in Section~\ref{subsec:Phi_identity_integral_proof}.
We also remark that~\eqref{eq:int_gauss_density_gauss_distr_funct_complex_arg} with $n=d+2$ can be viewed as a manifestation of the fact that $p_{d,-1}^{\text{Gauss}}= 0$: take formally $m=-1$ in~\eqref{eq:sylvester_refined_gaussian_formula}.


\begin{proof}[Proof of Theorem~\ref{theo:sylvester_probab_types_gauss}]
The proof relies on an explicit formula for $\E f_\ell(\cP_{d+2,d})$. To state it, we need some preparations.
Let $e_1,\ldots,e_n$ be the standard orthonormal basis in $\R^n$. Consider the $n$-vertex regular simplex
$
\Delta_n:=\conv(e_1,\ldots,e_n).
$
Note that $\dim \Delta_n = n-1$. Consider some $(k-1)$-dimensional face of $\Delta_n$, for example $\Delta_k:= \conv(e_1,\ldots, e_k)$. Here, $k\in \{1,\ldots, n\}$. By the definition of tangent cones given in Section~\ref{subsec:cones_angles_defs}, the tangent cone of $\Delta_n$ at $\Delta_k$ is given by
$$
T(\Delta_k, \Delta_n) = \pos\left(e_i- \frac {e_1+\ldots+e_k}k: i=1,\ldots, n\right).
$$
The internal angle of $\Delta_n$ at $\Delta_k$ is $\alpha (T(\Delta_k, \Delta_n))$, where $\alpha(C)$ denotes the angle of the polyhedral cone $C$. 
The internal angle sum of $\Delta_n$ at its $(k-1)$-dimensional faces is denoted by
$$
\bJ_{n,k}(\infty) =
\binom nk \cdot  \alpha \left(T(\Delta_k, \Delta_n)\right),
\quad k\in \{1,\ldots, n\}.
$$
The following explicit formulas for these angle sums are known: For all $n\in \N$ and $k\in \{1,\ldots, n\}$,
\begin{align}
\bJ_{n,k}(\infty)
&=
\binom nk \cdot \frac 1 {\sqrt {2\pi}} \int_{-\infty}^{+\infty} \Phi^{n-k} \left( \frac{\ii x}{\sqrt n}\right)  \eee^{-x^2/2} \dd x, \label{eq:regular_simpl_internal}
\end{align}
see the papers of Rogers~\cite[Section~4]{rogers} (where the method used was attributed to H.\ E.\ Daniels) as well as Vershik and Sporyshev~\cite[Lemma~4]{vershik_sporyshev_asymptotic_faces_random_polyhedra1992}. This formula and further  results can be found in~\cite[Proposition~1.2]{kabluchko_zaporozhets_absorption}.

Now, it follows from a formula due to~\citet{affentranger_schneider_random_proj} (taking into account~\cite{baryshnikov_vitale}) that
$$
\E f_\ell(\cP_{d+2,d})
=
\binom{d+2}{\ell+1} - 2 \cdot \bJ_{d+2,\ell+1} \left(\infty\right),
\qquad
\ell \in \{0,\ldots, d\}.
$$
With this formula at hand, we can apply Theorem~\ref{theo:sylvester_probab_types_trhough_expected_f_vectors} to our setting: For all $m\in \{0,1,\ldots, \lfloor d/2 \rfloor\}$,
\begin{align}
p_{d,m}^{\text{Gauss}}
&=
\eta_{d,m} \cdot \sum_{\ell=0}^{m} (-1)^{m+\ell} \binom{d-\ell+1}{d-m+1} \bJ_{d+2,\ell+1} \left(\infty\right) \notag
\\
&=
\eta_{d,m} \cdot  \sum_{\ell=0}^{m} (-1)^{m+\ell} \binom{d-\ell+1}{d-m+1}
\binom {d+2}{\ell+1} \cdot \frac 1 {\sqrt {2\pi}} \int_{-\infty}^{+\infty} \Phi^{d+1-\ell} \left( \frac{\ii x}{\sqrt {d+2}}\right)  \eee^{-x^2/2} \dd x  \notag
\\
&=
\frac {\eta_{d,m}} {\sqrt {2\pi}}  \cdot \int_{-\infty}^{+\infty} \left(\sum_{\ell=0}^{m} (-1)^{m+\ell} \binom{d-\ell+1}{d-m+1}
\binom {d+2}{\ell+1} \Phi^{d+1-\ell} \left( \frac{\ii x}{\sqrt {d+2}}\right)\right)  \eee^{-x^2/2} \dint x.
\label{eq:sylvester_refined_gauss_proof_auxiliary_1}
\end{align}
By applying Lemma~\ref{lem:probab_types_proof_identity_binomial_coeff}  with $z=\Phi(\frac{\ii x}{\sqrt {d+2}})$ and $1-z = \Phi(-\frac{\ii x}{\sqrt {d+2}})$ we transform~\eqref{eq:sylvester_refined_gauss_proof_auxiliary_1} to
\begin{align}
p_{d,m}^{\text{Gauss}}
&=
\frac {\eta_{d,m}} {\sqrt {2\pi}} \binom{d+2}{m+1} \label{eq:p_d_m_Gauss_complicated_formula}
\\
&\quad
\times
\int_{-\infty}^{+\infty}  \eee^{-x^2/2}
 \left(\Phi^{d+1-m} \left( \frac{\ii x}{\sqrt {d+2}}\right) \Phi^{m+1} \left(-\frac{\ii x}{\sqrt {d+2}}\right) + (-1)^m \Phi^{d+2} \left( \frac{\ii x}{\sqrt {d+2}}\right)\right)
 \dd x.  \notag
\end{align}
To complete the proof, it remains to observe that the contribution of the term $(-1)^m \Phi^{d+2}(\ldots)$ vanishes by~\eqref{eq:int_gauss_density_gauss_distr_funct_complex_arg}.
\end{proof}

\begin{corollary}[Sylvester problem for Gaussian points]
Let $X_1,\ldots, X_{d+2}$ be independent and standard Gaussian random points in $\R^d$. Then, the probability that $[X_1,\ldots, X_{d+2}]$ is a simplex is given by
$$
p_{d,0}^{\text{Gauss}} =  \frac{2(d+2)}{\sqrt{2\pi}} \int_{-\infty}^{+\infty} \eee^{-x^2/2} \Phi^{d+1}\left(\frac{\ii x}{\sqrt{d+2}}\right) \dint x
=
2 \cdot \bJ_{d+2,1} (\infty).
$$
\end{corollary}
\begin{proof}
Taking $m=0$  in Theorem~\ref{theo:sylvester_probab_types_gauss}, substituting $-x$ for $x$ and using the identity  $\Phi(z) + \Phi(-z) = 1$ gives
\begin{align*}
p_{d,0}^{\text{Gauss}}
&=
\frac{2(d+2)}{\sqrt{2\pi}} \int_{-\infty}^{+\infty} \eee^{-x^2/2} \Phi\left(-\frac{\ii x}{\sqrt{d+2}}\right) \Phi^{d+1}\left(\frac{\ii x}{\sqrt{d+2}}\right) \dint x
\\
&=
\frac{2(d+2)}{\sqrt{2\pi}} \int_{-\infty}^{+\infty} \eee^{-x^2/2} \left(1 - \Phi\left(\frac{\ii x}{\sqrt{d+2}}\right)\right) \Phi^{d+1}\left(\frac{\ii x}{\sqrt{d+2}}\right) \dint x
\\
&=
\frac{2(d+2)}{\sqrt{2\pi}} \int_{-\infty}^{+\infty} \eee^{-x^2/2} \Phi^{d+1}\left(\frac{\ii x}{\sqrt{d+2}}\right) \dint x,
\end{align*}
where the last line follows from~\eqref{eq:int_gauss_density_gauss_distr_funct_complex_arg}. The right-hand side equals $2 \cdot \bJ_{d+2,1} (\infty)$ by~\eqref{eq:regular_simpl_internal}.  A purely geometric proof of the identity $p_{d,0}^{\text{Gauss}} = 2 \cdot \bJ_{d+2,1} (\infty)$ was given in~\cite{gusakova_kabluchko_sylvester_beta}.
\end{proof}

Recently, the ultra log-concavity of the sequence $m\mapsto p_{d,m}^{\text{Gauss}}$ has been established in~\cite{ChanKalaiNarayananTerSaakovWhite2025UnimodalityRadon}.

\subsection{Youden's problem}\label{subsec:youden_problem}
Let $\xi_1,\ldots, \xi_{n}\sim {\rm N}(0,1)$ be independent standard Gaussian random variables with order statistics $\xi_{1:n} \leq \ldots \leq \xi_{n:n}$ and sample mean $\bar{\xi} = \frac 1 n (\xi_1 + \ldots + \xi_n)$. Youden's demon problem~\cite{youden_sets_of_three_measurements,kendall_two_problems_youden,david_sample_mean_among_extreme,david_sample_mean_moderate_order,dmitrienko_demon,dmitrienko_demon_exponential,frick_newman_pegden} asks to determine the probability $P(n,k)$ that $\bar \xi$ lies between $\xi_{k:n}$ and $\xi_{k+1:n}$, for $k=1,\ldots, n-1$. \citet{frick_newman_pegden} related this problem to the Sylvester problem for Gaussian points.  Their result reads as follows.
\begin{theorem}[Frick, Newman, Pegden~\cite{frick_newman_pegden}] \label{theo:frick_newman_pegden}
For all $m \in \{0,\ldots, \lfloor d/2 \rfloor\}$ we have
\begin{equation}\label{eq:frick_newman_pegden}
p_{d,m}^{\text{Gauss}} =  \eta_{d,m} \cdot P(d+2,m+1)
=
\begin{cases}
2 P(d+2,m+1),&\text{ if } m\neq d/2,\\
P(d+2,m+1), &\text{ if } m=d/2.
\end{cases}
\end{equation}
\end{theorem}


Combining this result with the formula for $p_{d,m}^{\text{Gauss}}$ obtained in Theorem~\ref{theo:sylvester_probab_types_gauss}, we can recover an expression for $P(n,k)$ that is equivalent to Kuchelmeister's recent solution \cite{kuchelmeister_youdens_demon} of Youden's problem.
\begin{theorem}[Kuchelmeister~{\cite[Corollary 3]{kuchelmeister_youdens_demon}}]\label{theo:youden_probability}
Consider independent standard  normal random variables $\xi_1,\ldots, \xi_n$, where $n\geq 2$.  Then, for all $k\in \{1,\ldots, n-1\}$,
$$
P(n,k) := \P[ \xi_{k:n}\leq \bar \xi \leq \xi_{k+1:n}] = \frac{1}{\sqrt{2\pi}} \binom{n}{k} \int_{-\infty}^{+\infty} \eee^{-x^2/2} \Phi^{k}\left(\frac{\ii x}{\sqrt{n}}\right) \Phi^{n-k}\left(-\frac{\ii x}{\sqrt{n}}\right) \dint x.
$$
\end{theorem}
\begin{proof}
Let first $n\geq 3$.  Then, \eqref{eq:frick_newman_pegden} with $d := n-2 \in \N$ and $m := k-1\in \{0,\ldots, d\}$ gives
$$
P(n,k) = \eta_{d, m}^{-1} \cdot  p_{d,m}^{\text{Gauss}}.
$$
If $m\leq \lfloor d/2 \rfloor$, the claimed formula follows from Theorem~\ref{theo:sylvester_probab_types_gauss}. For $m>\lfloor d/2 \rfloor$, it follows from the symmetry property $P(n,k) = P(n, n-k)$.

In case $n=2$, it is clear that $P(2,1)=1$, and we can use the identity  $\Phi(z) + \Phi(-z) = 1$ (twice) along with \eqref{eq:int_gauss_density_gauss_distr_funct_complex_arg} to evaluate the right-hand side as
\begin{align*}
\frac{2}{\sqrt{2\pi}} \int_{-\infty}^{+\infty} \eee^{-x^2/2} \Phi\left(\frac{\ii x}{\sqrt{2}}\right)\Phi\left(-\frac{\ii x}{\sqrt{2}}\right)\dint x &=\frac{2}{\sqrt{2\pi}} \int_{-\infty}^{+\infty} \eee^{-x^2/2} \Phi\left(\frac{\ii x}{\sqrt{2}}\right)\left(1- \Phi\left(\frac{\ii x}{\sqrt{2}}\right) \right)\dint x \\
&=\frac{2}{\sqrt{2\pi}} \int_{-\infty}^{+\infty} \eee^{-x^2/2} \Phi\left(\frac{\ii x}{\sqrt{2}}\right)\dint x\\
&=\frac{2}{\sqrt{2\pi}} \int_0^{+\infty} \eee^{-x^2/2} \left(\Phi\left(\frac{\ii x}{\sqrt{2}}\right)+\Phi\left(-\frac{\ii x}{\sqrt{2}}\right)\right)\dint x\\
&=\frac{2}{\sqrt{2\pi}} \int_0^{+\infty} \eee^{-x^2/2} \dint x=1.
\end{align*}
\end{proof}

\begin{remark}
The formula holds also for $k=0$ and $k=n$. Indeed, $P(n,0)= P(n,n) = 0$ since it is not possible that all $\xi_i$ are smaller/larger than their mean, while the right-hand side also becomes $0$ by~\eqref{eq:int_gauss_density_gauss_distr_funct_complex_arg}.
\end{remark}
\begin{remark}
Let us provide a heuristic explanation of Theorem~\ref{theo:youden_probability}. The random vector $(\xi_1-\bar \xi, \ldots, \xi_n -\bar \xi)$ is $n$-variate Gaussian with zero mean and covariance matrix
$$
\E [(\xi_i - \bar \xi) (\xi_j - \bar \xi)] =  \delta_{ij} - \frac 1n - \frac 1n + \frac 1n =  \delta_{ij} - \frac 1n, \qquad i,j \in \{1,\ldots, n\},
$$
where $\delta_{ij}$ is the Kronecker delta.
On the other hand, consider an $n$-variate Gaussian vector $(\eta_1 - \sigma \eta,\ldots, \eta_n - \sigma \eta)$, where $\eta, \eta_1,\ldots,\eta_n \sim {\rm N}(0,1)$ are independent standard Gaussian random variables and $\sigma$ is a constant. The covariance matrix of this vector is given by
$$
\E[(\eta_i - \sigma \eta)(\eta_j - \sigma \eta)] = \delta_{ij} + \sigma^2, \qquad i,j \in \{1,\ldots, n\}.
$$
If we choose $\sigma^2 = -1/n$ and ignore the fact that the variance cannot be negative, we can identify the random vector $(\xi_i - \bar \xi)_{i=1}^n$ with $(\eta_i - \sigma \eta)_{i=1}^n$. This leads to
\begin{align}
P(n,k)
&=
\binom nk \cdot \P[\xi_1 - \bar \xi \leq 0, \ldots, \xi_k - \bar \xi \leq 0, \xi_{k+1} - \bar \xi  \geq 0,\ldots,\xi_{n} - \bar \xi  \geq 0]
\\
&=
\binom nk \cdot \P[\eta_1 \leq \sigma \eta, \ldots,  \eta_k \leq \sigma \eta, \eta_{k+1} \geq \sigma \eta, \ldots, \eta_n \geq \sigma \eta].
\end{align}
Conditioning on $\eta = x$,  and integrating over $x$, we obtain
$$
P(n,k)
=
\binom nk \cdot \frac 1 {\sqrt {2\pi}} \int_{-\infty}^{+\infty} \eee^{-x^2/2} \Phi^{k}\left(\sigma x\right) \Phi^{n-k}\left(-\sigma x\right) \dint x.
$$
Now recall that $\sigma^2 = -1/n$. Inserting $\sigma = \ii/\sqrt n$ we obtain the claimed formula. Note that $\sigma = - \ii/\sqrt n$ would give the same formula after substituting $x$ by $-x$ in the integral.
\end{remark}



\subsection{Beta-type random points}\label{subsec:beta_type_sylvester_refined}
In this section, we shall solve the refined Sylvester problem for i.i.d.\ points following a beta distribution or a beta prime distribution. These are defined as follows.
A random point in $\R^d$ is said to follow a $d$-dimensional \textit{beta distribution} with parameter $\beta>-1$ if its Lebesgue density is given by
\begin{equation}\label{eq:def_f_beta}
f_{d,\beta}(x)=c_{d,\beta} \left( 1-\left\| x \right\|^2 \right)^\beta\ind_{\{\|x\| <  1\}},\qquad x\in\R^d,\qquad
c_{d,\beta}= \frac{ \Gamma\left( \frac{d}{2} + \beta + 1 \right) }{ \pi^{ \frac{d}{2} } \Gamma\left( \beta+1 \right) }.
\end{equation}
Here, $\|x\| = (x_1^2+\dots +x_d^2)^{1/2}$ denotes the Euclidean norm of the vector $x = (x_1,\dots,x_d)\in \mathbb{R}^d$. The uniform distribution on the unit ball $\mathbb{B}^d$ is recovered by taking $\beta = 0$, whereas the uniform distribution on the unit sphere $\mathbb{S}^{d-1}$ is the weak limit of the beta distribution as $\beta\downarrow -1$. Accordingly, when we refer to random points in $\mathbb{R}^d$ with the ``beta density $f_{d,-1}$'', what we really mean is that they are uniform on $\mathbb{S}^{d-1}$.

A random point in $\R^d$ follows a $d$-dimensional \textit{beta prime distribution} with parameter $\beta>d/2$ if its Lebesgue density is given by
\begin{equation}\label{eq:def_f_beta_prime}
\widetilde{f}_{d,\beta}(x)=\widetilde{c}_{d,\beta} \left( 1+\left\| x \right\|^2 \right)^{-\beta},\qquad
x\in\R^d,\qquad
\widetilde{c}_{d,\beta}= \frac{ \Gamma\left( \beta \right) }{\pi^{ \frac{d}{2} } \Gamma\left( \beta - \frac{d}{2} \right)}.
\end{equation}
These distributions were introduced to stochastic geometry by~\citet{miles} and~\citet{ruben_miles}.  The beta, beta-prime, and isotropic normal distributions enjoy two key properties, invariance under projections and invariance under slicing~\cite[Chapter~2]{kabluchko_steigenberger_thaele_boob_beta_type}, which make them particularly amenable to computations in stochastic geometry; see~\cite{kabluchko_steigenberger_thaele_boob_beta_type}. We shall use the notation
\begin{equation}\label{eq:c_beta}
c_{\beta}
:=
c_{1,\beta}
=
\frac{ \Gamma\left(\beta + \frac{3}{2} \right) }{  \sqrt \pi\; \Gamma (\beta+1)},
\qquad
\widetilde c_{\beta}
:=
\widetilde{c}_{1,\beta}
=
\frac{ \Gamma(\beta)}{ \sqrt \pi\; \Gamma\left( \beta - \frac{1}{2}\right)}.
\end{equation}

\begin{theorem}[Probabilities of types for beta-distributed points]\label{theo:sylvester_probab_types_beta}
Let $X_1,\ldots, X_{d+2}$ be independent random points in $\R^d$ with beta density $f_{d,\beta}$, where $d\geq 2$ and $\beta \geq -1$.  Consider their convex hull $\cP_{d+2,d}^{\beta} = [X_1,\ldots, X_{d+2}]$. Then, for all $m\in \{0,1,\ldots, \lfloor d/2 \rfloor\}$,
\begin{align*}
p_{d,m}(\beta)
:&=
\P[\cP_{d+2,d}^{\beta} \text{ is of type } T_m^{d}]
\\
&=
\eta_{d,m} \binom {d+2}{m+1} \int_{-\infty}^{+\infty} \frac{c_{\frac{(2\beta +d)(d+2)}{2}} \left((G(x))^{d+1-m} (1-G(x))^{m+1} + (-1)^m (G(x))^{d+2}\right)}{(\cosh x)^{(2\beta+d)(d+2) +2}}
 \dint x,
\end{align*}
where
\begin{equation}\label{eq:beta_distr_G_2beta+d}
G(x) = G_{2\beta + d} (x) =  \frac 12  + \ii \int_0^x  c_{\frac{2\beta+d-1}{2}} (\cosh y)^{2\beta + d}\dd y.
\end{equation}
\end{theorem}
\begin{proof}
The proof relies on a formula for $\E f_\ell(\cP_{d+2,d}^\beta)$ obtained in~\cite{beta_polytopes} and~\cite{kabluchko_angles_explicit_formula}. To state it, we need to introduce some notation.
Let $Z_1,\ldots,Z_{n}$ be i.i.d.\ random points in $\R^{n-1}$ following the beta distribution $f_{n-1,\beta}$, where $\beta \geq  -1$.   The convex hull $[Z_1,\ldots,Z_n]$ is called an $(n-1)$-dimensional \textit{beta simplex}.
By the definition given in Section~\ref{subsec:cones_angles_defs},  the tangent cone of $[Z_1,\ldots,Z_n]$ at its face $[Z_1,\ldots, Z_k]$ is given by
$$
T([Z_1,\ldots, Z_k], [Z_1,\ldots, Z_n]) = \pos\left\{Z_i - \frac{Z_1+\ldots + Z_k}{k}: i=1,\ldots, n\right\}.
$$
The expected sum of internal angles at $k$-vertex faces of $[Z_1,\ldots, Z_n]$ is denoted by
\begin{equation}
\bJ_{n,k}(\beta) := \binom {n}{k} \cdot  \E \alpha\left(T([Z_1,\ldots, Z_k], [Z_1,\ldots, Z_n])\right),
\end{equation}
for all integer $n\geq 2$ and $k\in \{1,\ldots,n\}$.  An explicit formula for the quantities $\bJ_{n,k}(\beta)$ has been obtained in~\cite[Theorem~1.2 and Eqn.~(1.7)]{kabluchko_angles_explicit_formula}:
\begin{equation}\label{eq:J_nk_integral}
\bJ_{n,k}\left(\frac{\alpha - n + 1}{2}\right)
=
\binom nk \int_{-\infty}^{+\infty} c_{\frac{\alpha n}2} (\cosh x)^{-\alpha n - 2}
\left(\frac 12  + \ii \int_0^x  c_{\frac{\alpha-1}{2}} (\cosh y)^{\alpha}\dd y \right)^{n-k} \dd x,
\end{equation}
for all integer $n\geq 3$, $k\in \{1,\ldots,n\}$ and $\alpha \geq n-3$.

The formula for the expected $f$-vector of $\cP_{d+2,d}^\beta$ given in~\cite[Theorem 1.2]{beta_polytopes}  simplifies (for $n=d+2$ points and taking into account relations stated in~\cite[Proposition~2.1]{kabluchko_recursive_scheme}) to
$$
\E f_\ell(\cP_{d+2,d}^\beta)
=
\binom{d+2}{\ell+1} - 2 \cdot \bJ_{d+2,\ell+1} \left(\beta - \frac 12\right), \qquad \beta \geq -\frac 12.
$$
So, let us first suppose that $\beta \geq -1/2$. It follows from~\eqref{eq:J_nk_integral} with $n=d+2$ and $\alpha = 2\beta + d$ that
\begin{align*}
\binom{d+2}{\ell+1} - \E f_\ell(\cP_{d+2,d}^\beta)
&=
2\cdot \bJ_{d+2,\ell+1} \left(\beta - \frac 12\right)
\\
&=
2\cdot \binom {d+2}{\ell+1} \int_{-\infty}^{+\infty} \frac{c_{\frac{(2\beta +d)(d+2)}{2}}}{(\cosh x)^{(2\beta+d)(d+2) + 2}}
\left( G (x)\right)^{d+1-\ell} \dd x,
\end{align*}
where $G(x)$ is as in~\eqref{eq:beta_distr_G_2beta+d}.
This formula, combined with Theorem~\ref{theo:sylvester_probab_types_trhough_expected_f_vectors} gives that for all $m\in \{0,1,\ldots, \lfloor d/2 \rfloor\}$,
\begin{align*}
p_{d,m}(\beta)
&=
\frac 12 \eta_{d,m} \cdot \sum_{\ell=0}^{m} (-1)^{m+\ell} \binom{d-\ell+1}{d-m+1}
\left(\binom{d+2}{\ell+1} - \E f_\ell(\cP_{d+2,d}^\beta)\right)
\\
&=
\eta_{d,m} \cdot \sum_{\ell=0}^{m} (-1)^{m+\ell} \binom{d-\ell+1}{d-m+1}
\binom {d+2}{\ell+1}
\int_{-\infty}^{+\infty} \frac{c_{\frac{(2\beta +d)(d+2)}{2}}}{(\cosh x)^{(2\beta+d)(d+2) + 2}}
\left( G (x)\right)^{d+1-\ell} \dd x
\\
&=
\eta_{d,m} \cdot
\int_{-\infty}^{+\infty} \frac{c_{\frac{(2\beta +d)(d+2)}{2}}}{(\cosh x)^{(2\beta+d)(d+2) + 2}}
\left(\sum_{\ell=0}^{m} (-1)^{m+\ell} \binom{d-\ell+1}{d-m+1}
\binom {d+2}{\ell+1} \left(G (x)\right)^{d+1-\ell}\right) \dd x.
\end{align*}
Transforming the sum by means of Lemma~\ref{lem:probab_types_proof_identity_binomial_coeff} gives
\begin{align*}
p_{d,m}(\beta)
&=
\eta_{d,m} \cdot
\binom{d+2}{m+1} \int_{-\infty}^{+\infty} \frac{c_{\frac{(2\beta +d)(d+2)}{2}}}{(\cosh x)^{(2\beta+d)(d+2) + 2}}
\left((G(x))^{d+1-m} (1-G(x))^{m+1} + (-1)^m (G(x))^{d+2}\right)
\dd x.
\end{align*}
To complete the proof of the theorem for $\beta\geq -1/2$, observe that  $1-G(x) = G(-x)$  by~\eqref{eq:beta_distr_G_2beta+d}.

To extend the result to the range $\beta > -1$, we argue by analytic continuation as in~\cite[Remark~2.5]{gusakova_kabluchko_sylvester_beta}. The function $p_{d,m}(\beta)$, originally defined for $\beta> -1$,  admits an analytic continuation to the half-plane $\{\Re \beta >-1\}$. To see this, write $p_{d,m}(\beta) = \int_{D_m} f_{d,\beta}(x_1) \ldots f_{d,\beta} (x_{d+2}) \dint x_1 \dots \dint x_{d+2}$, where $D_m = \{(x_1,\dots, x_{d+2})\in (\bB^d)^{d+2}: [x_1,\ldots, x_{d+2}] \text{ is of type } T_{d}^m\}$, and apply Lemma~4.3 in~\cite{beta_polytopes}. On the other hand, the double integral in Theorem~\ref{theo:sylvester_probab_types_beta} defines an analytic function of $\beta$ on $\{\Re \beta > -\frac d2 - \frac 1{d+2}\}$. By the identity theorem for analytic functions, the double-integral formula for $p_{d,m}(\beta)$ continues to hold for all $\beta >-1$. The case $\beta= -1$ follows by continuity.
\end{proof}

\begin{remark}
The formula for $p_{d,m}(\beta)$ stated in Theorem~\ref{theo:sylvester_probab_types_beta} is structurally similar to formula~\eqref{eq:p_d_m_Gauss_complicated_formula}. We were able to simplify the latter formula using~\eqref{eq:int_gauss_density_gauss_distr_funct_complex_arg}.
Unfortunately, an
analogue of~\eqref{eq:int_gauss_density_gauss_distr_funct_complex_arg} is not true in the beta setting: In general,
$\int_{-\infty}^{+\infty} \frac{(G(x))^{d+2}}{(\cosh x)^{(2\beta+d)(d+2) + 2}}
\dd x$ need not vanish.   
\end{remark}

\begin{example}[Uniform distribution on the sphere]
Let $X_1,\ldots, X_{d+2}$ be i.i.d.\ points uniformly distributed on the unit sphere $\bS^{d-1}$, which corresponds to the case $\beta=-1$.
The usual Sylvester problem becomes trivial: we have $p_{d,0}(-1) = 0$ since no point $X_i$ can be inside the convex hull of the remaining points. The refined Sylvester problem, however, is non-trivial. The values of $p_{d,\bullet}(-1) = (p_{d,m}(-1))_{m=0}^{\lfloor d/2\rfloor}$ for small $d$ are given in the following table:
\begin{align*}
p_{3, \bullet}(-1) &= \left(0, 1\right),
\\
p_{4, \bullet}(-1) &= \left(0,\frac{3289}{240 \pi ^2}-1,2-\frac{3289}{240 \pi ^2}\right),
\\
p_{5, \bullet}(-1) &= \left(0,\frac{81}{646},\frac{565}{646}\right),
\\
p_{6, \bullet}(-1) &= \left(0,\frac{4}{3}+\frac{3033926553841}{92461824000 \pi ^4}-\frac{14969498023}{927566640 \pi ^2},-\frac{3033926553841}{23115456000 \pi ^4}-8+\frac{14969498023}{154594440 \pi ^2},\right.\\
&\qquad \qquad  \qquad \left.\frac{23}{3}+\frac{3033926553841}{30820608000 \pi ^4}-\frac{14969498023}{185513328 \pi ^2}\right),
\\
p_{7, \bullet}(-1) &=\left(0,\frac{510000}{58642669},\frac{11730250}{58642669},\frac{46402419}{58642669}\right).
\end{align*}
\end{example}

\begin{example}[Uniform distribution in the ball]
Let $X_1,\ldots, X_{d+2}$ be i.i.d.\ points uniformly distributed on the unit ball $\bB^{d}$, which corresponds to the case $\beta=0$. The formula for $p_{d,0}(0)$ is due to Kingman; see~\eqref{eq:kingman_sylvester_simplex}. The values of $p_{d,\bullet}(0) = (p_{d,m}(0))_{m=0}^{\lfloor d/2\rfloor}$ for small $d$ are given in the following table:
\begin{align*}
p_{3, \bullet}(0) &= \left(\frac{9}{143}, \frac{134}{143}\right),
\\
p_{4, \bullet}(0) &=\left(\frac{676039}{648000 \pi ^4},-\frac{676039}{162000 \pi ^4}-1+\frac{13453919}{931392 \pi ^2},\frac{676039}{216000 \pi ^4}+2-\frac{13453919}{931392 \pi ^2}\right),
\\
p_{5, \bullet}(0) &= \left(\frac{20000}{12964479},\frac{20500}{139403},\frac{11037979}{12964479}\right).
\end{align*}

\end{example}

\begin{theorem}[Probabilities of types for beta-prime-distributed points]\label{theo:sylvester_probab_types_beta_prime}
Let $\widetilde X_1,\ldots, \widetilde X_{d+2}$ be independent random points in $\R^d$ with beta prime density $\widetilde f_{d,\beta}$, where $2\beta> d + \frac 1 {d+2}$.     Consider their convex hull $\widetilde \cP_{d+2,d}^{\beta} = [\widetilde X_1,\ldots, \widetilde X_{d+2}]$. Then, for all $m\in \{0,1,\ldots, \lfloor d/2 \rfloor\}$,
\begin{align*}
\widetilde p_{d,m}(\beta)
:&=
\P[\widetilde{\cP}_{d+2,d}^{\beta} \text{ is of type } T_m^{d}]
\\
&=
\eta_{d,m} \cdot
\binom{d+2}{m+1} \int_{-\infty}^{+\infty} \frac{\widetilde c_{\frac{(2\beta - d)(d+2)}{2}} \left((\widetilde G(x))^{d+1-m} (1-\widetilde G(x))^{m+1} + (-1)^m (\widetilde  G(x))^{d+2}\right)}{(\cosh x)^{(2\beta - d)(d+2) - 1}}
\dd x,
\end{align*}
where
\begin{equation}\label{eq:beta_distr_G_tilde_2beta-d}
\widetilde G(x) = \widetilde G_{2\beta-d} (x) =  \frac 12  + \ii \int_0^x  \widetilde c_{\frac{2\beta-d+1}{2}} (\cosh y)^{2\beta - d - 1}\dd y.
\end{equation}
\end{theorem}
\begin{proof}

Similarly to the beta case, let $\widetilde Z_1,\ldots,\widetilde Z_{n}$ be independent random points in $\R^{n-1}$ sampled from the beta prime distribution $\widetilde f_{n-1,\beta}$, where $\beta > (n-1)/2$. Their convex hull $[\widetilde Z_1,\ldots,\widetilde Z_n]$ is called the $(n-1)$-dimensional \textit{beta prime simplex}. Its expected internal angles sum at all $k$-vertex faces is denoted by
\begin{equation}
\widetilde \bJ_{n,k}(\beta)  := \binom {n}{k} \cdot  \E \alpha\left(\pos\left\{\widetilde Z_i - \frac{\widetilde Z_1+\ldots + \widetilde Z_k}{k}: i=1,\ldots, n\right\}\right).
\end{equation}
for all integer $n\geq 2$ and $k\in \{1,\ldots,n\}$.  
An explicit formula for the quantities $\widetilde \bJ_{n,k}(\beta)$ has been derived in~\cite[Theorem~1.7 and Eqn.~(1.16)]{kabluchko_angles_explicit_formula}: For all $n\in\N$, $k\in \{1,\ldots,n\}$  and $\alpha > 1/n$ we have
\begin{equation}\label{eq:J_nk_tilde_integral}
\widetilde \bJ_{n,k}\left(\frac{\alpha + n - 1}{2}\right)
=
\binom nk \int_{-\infty}^{+\infty} \widetilde c_{\frac{\alpha n}2} (\cosh x)^{-(\alpha n - 1)} \left(\frac 12  + \ii \int_0^x \widetilde c_{\frac{\alpha+1}{2}}(\cosh y)^{\alpha-1}\dd y \right)^{n-k} \dd x.
\end{equation}

The formula for the expected $f$-vector of $\widetilde \cP_{d+2,d}^\beta$ given in~\cite[Theorem 1.14]{beta_polytopes}  simplifies (for $n=d+2$ points and taking into account relations stated in~\cite[Proposition~2.1]{kabluchko_recursive_scheme}) to

\begin{align*}
\E f_\ell( \widetilde \cP_{d+2,d}^\beta)
&=
\binom{d+2}{\ell+1} - 2 \cdot \widetilde \bJ_{d+2,\ell+1} \left(\beta + \frac 12\right),  \qquad \beta > \frac d2.
\end{align*}
It follows from~\eqref{eq:J_nk_tilde_integral} with $n=d+2$ and $\alpha = 2\beta - d$ that
\begin{align*}
\binom{d+2}{\ell+1} - \E f_\ell(\widetilde \cP_{d+2,d}^\beta)
&=
2\cdot \widetilde \bJ_{d+2,\ell+1} \left(\beta + \frac 12\right)
\\
&=
2\cdot \binom {d+2}{\ell+1} \int_{-\infty}^{+\infty} \frac{\widetilde c_{\frac{(2\beta - d)(d+2)}{2}}}{(\cosh x)^{(2\beta - d)(d+2) - 1}}
( \widetilde G (x))^{d+1-\ell} \dd x,
\end{align*}
where $\widetilde G(x)$ is as in~\eqref{eq:beta_distr_G_tilde_2beta-d} and the last formula requires the condition $2\beta> d + \frac 1 {d+2}$. Combining this formula with Theorem~\ref{theo:sylvester_probab_types_trhough_expected_f_vectors} gives that for all $m\in \{0,1,\ldots, \lfloor d/2 \rfloor\}$,
\begin{align*}
\widetilde p_{d,m}(\beta)
&=
\frac 12 \eta_{d,m} \cdot \sum_{\ell=0}^{m} (-1)^{m+\ell} \binom{d-\ell+1}{d-m+1}
\left(\binom{d+2}{\ell+1} - \E f_\ell(\widetilde \cP_{d+2,d}^\beta)\right)
\\
&=
\eta_{d,m} \cdot \sum_{\ell=0}^{m} (-1)^{m+\ell} \binom{d-\ell+1}{d-m+1}
\binom {d+2}{\ell+1}
\int_{-\infty}^{+\infty} \frac{\widetilde c_{\frac{(2\beta-d)(d+2)}{2}}}{(\cosh x)^{(2\beta-d)(d+2) - 1}}
\left( \widetilde G (x)\right)^{d+1-\ell} \dd x
\\
&=
\eta_{d,m} \cdot
\int_{-\infty}^{+\infty} \frac{\widetilde c_{\frac{(2\beta - d)(d+2)}{2}}}{(\cosh x)^{(2\beta-d)(d+2) - 1}}
\left(\sum_{\ell=0}^{m} (-1)^{m+\ell} \binom{d-\ell+1}{d-m+1}
\binom {d+2}{\ell+1} (\widetilde G (x))^{d+1-\ell}\right) \dd x.
\end{align*}
Transforming the sum by means of Lemma~\ref{lem:probab_types_proof_identity_binomial_coeff} gives
\begin{align*}
\widetilde p_{d,m}(\beta)
&=
\eta_{d,m} \cdot
\binom{d+2}{m+1} \int_{-\infty}^{+\infty} \frac{\widetilde c_{\frac{(2\beta - d)(d+2)}{2}}}{(\cosh x)^{(2\beta - d)(d+2) - 1}}
\left((\widetilde G(x))^{d+1-m} (1-\widetilde G(x))^{m+1} + (-1)^m (\widetilde  G(x))^{d+2}\right)
\dd x.
\end{align*}
To complete the proof, note that $1-\widetilde G(x) = \widetilde G(-x)$ by~\eqref{eq:beta_distr_G_tilde_2beta-d}.
\end{proof}
\begin{remark}
The formula for $\widetilde p_{d,m}(\beta)$ stated in Theorem~\ref{theo:sylvester_probab_types_beta_prime} is structurally similar to formula~\eqref{eq:p_d_m_Gauss_complicated_formula}, which was simplified using~\eqref{eq:int_gauss_density_gauss_distr_funct_complex_arg}. Unfortunately, a beta prime  analogue of~\eqref{eq:int_gauss_density_gauss_distr_funct_complex_arg} is not true: In general,
$\int_{-\infty}^{+\infty} \frac{(\widetilde G(x))^{d+2}}{(\cosh x)^{(2\beta - d)(d+2) - 1}}
\dd x$ need not vanish.
\end{remark}

\begin{example}[Sylvester problem for beta-type distributions]
For $m=0$ Theorems~\ref{theo:sylvester_probab_types_beta} and~\ref{theo:sylvester_probab_types_beta_prime} give the formulas $p_{d,0}(\beta) =  2 \cdot \bJ_{d+2,1} (\beta - \frac 12)$ and $\widetilde p_{d,0}(\beta) = 2 \cdot \widetilde \bJ_{d+2,1} (\beta + \frac 12)$ obtained in~\cite{gusakova_kabluchko_sylvester_beta}.
\end{example}

\begin{example}[Cauchy distribution]\label{ex:refined_sylvester_cauchy_distr}
The beta prime distribution $\tilde f_{d,\beta}$ with $\beta = \frac{d+1}2$ is the $d$-dimensional Cauchy distribution. In this case,  $\widetilde{G}(x) = \frac 12 + \frac{\ii}{\pi} x$ and Theorem~\ref{theo:sylvester_probab_types_beta_prime} gives
\begin{equation}\label{eq:refined_sylvester_for_cauchy_distr}
\widetilde p_{d,m}\left(\frac{d+1}{2}\right) =
\eta_{d,m} \cdot
\binom{d+2}{m+1} \frac{\widetilde c_{\frac{d+2}{2}}}{\pi^{d+2}}
\int_{-\infty}^{+\infty}
\frac{\left(\frac \pi2 + \ii x\right)^{d+2}}{(\cosh x)^{d+1}}
\left(\left(\frac{\frac \pi 2 - \ii x}{\frac \pi 2 + \ii x}\right)^{m+1} + (-1)^m\right)
\dd x.
\end{equation}
For small dimensions $d$ this gives
\begin{align*}
\widetilde p_{3, \bullet}(2) &= \left(\frac{5}{\pi ^2}-\frac{1}{3},\frac{4}{3}-\frac{5}{\pi ^2}\right),
\\
\widetilde p_{4, \bullet}(5/2) &= \left(\frac{80}{\pi ^4}+6-\frac{200}{3 \pi ^2},-\frac{320}{\pi ^4}-30+\frac{1000}{3 \pi ^2},\frac{240}{\pi ^4}+25-\frac{800}{3 \pi ^2}\right),
\\
\widetilde p_{5, \bullet}(3) &= \left(\frac{1}{3}+\frac{105}{8 \pi ^4}-\frac{35}{8 \pi ^2},-\frac{315}{8 \pi ^4}-2+\frac{105}{4 \pi ^2},\frac{8}{3}+\frac{105}{4 \pi ^4}-\frac{175}{8 \pi ^2}\right),
\\
\widetilde p_{6, \bullet}(7/2) &=\left(-\frac{1568}{3 \pi ^4}+\frac{896}{5 \pi ^6}-34+\frac{29008}{75 \pi ^2},\frac{10976}{3 \pi ^4}-\frac{5376}{5 \pi ^6}+238-\frac{203056}{75 \pi ^2},\right.\\
&\qquad \qquad \qquad \left.-\frac{29792}{3 \pi ^4}+\frac{2688}{\pi ^6}-658+\frac{1682464}{225 \pi ^2},\frac{20384}{3 \pi ^4}-\frac{1792}{\pi ^6}+455-\frac{232064}{45 \pi ^2}\right).
\end{align*}
The Cauchy distribution on $\R^d$ can be mapped to the uniform distribution on the upper half-sphere $\bS^d_+ = \{(x_0,\ldots, x_d)\in \bS^d: x_0\geq 0\}$ via the gnomonic projection; see, e.g., \cite[Section~2.2.1]{convex_hull_sphere}. After gnomonic projection, the result becomes:  If $U_1,\ldots, U_{d+2}$ are independent and uniformly distributed on $\bS^d_+$, then their spherical hull is of type $T_{m}^d$ (meaning that it has $(m+1)(d+1-m)$ facets) with the  probability given in~\eqref{eq:refined_sylvester_for_cauchy_distr}. This refines the analogue of the Sylvester problem on $\bS^d_+$ which has been solved in~\cite[Section~2.4]{kabluchko_poisson_zero}.
\end{example}


\subsection{Convex hulls of random walks}\label{subsec:convex_hulls_RW_sylvester_refined}
Let $\xi_1,\dots,\xi_n$ be (possibly dependent) random $d$-dimensional vectors with partial sums
$$
S_i = \xi_1 + \dots + \xi_i,\quad  1\leq i\leq n,\quad  S_0=0.
$$
The sequence $S_0,S_1,\dots,S_n$ will be referred to as \emph{random walk} or, if additionally the condition $S_n=0$ holds, a \emph{random bridge}.  Consider its convex hull
\begin{align*}
\cQ_{n,d} &:= \conv(S_0,S_1,\ldots, S_n).
\end{align*}
To state our result, we need to impose the following assumptions on the joint distribution of  the increments. 
\begin{enumerate}
\item[$(\text{Ex})$] \textit{Exchangeability:} For every permutation $\sigma= (\sigma(1),\ldots, \sigma(n))$ of the set $\{1,\ldots,n\}$, we have the distributional equality
$$
(\xi_{\sigma(1)},\ldots,  \xi_{\sigma(n)}) \eqdistr (\xi_1,\ldots,\xi_n).
$$
\item[$(\text{GP})$] \textit{General position:}
For every $1\leq i_1 < \ldots < i_d\leq n$, the probability that the vectors $S_{i_1}, \ldots,S_{i_d}$ are linearly dependent is $0$.
\end{enumerate}

Here we shall be interested in the case when $n= d+1$. The next theorem provides a formula for the probabilities of types of $\cQ_{d+1,d}$ in terms of the Eulerian numbers.

\begin{theorem}[Probabilities of types for convex hulls of random walks]\label{theo:sylvester_probab_types_conv_RW}
Under the above assumptions, for all $m\in \{0,1,\ldots, \lfloor d/2 \rfloor\}$,
$$
p_{d,m}^{\text{convRW}} := \P[\cQ_{d+1,d} \text{ is of type } T_m^{d}]
=
\eta_{d,m} \cdot \frac{\eulerian{d+1}{m}}{(d+1)!},
$$
where  $\eulerian{n}{k}$ is an Eulerian number counting permutations of $\{1,\ldots, n\}$ in which exactly $k$ elements are greater than the previous element, $k\in \{0,\ldots, n-1\}$.
\end{theorem}
\begin{remark}
For properties of Eulerian numbers we refer to the books of~\citet[Section~6.2]{graham_knuth_patashnik_book}, \citet[Sections~6.3-6.7]{mezo_book}, \citet{petersen_book_eulerian_numbers} and the paper of~\citet{janson_euler_frobenius_rounding}. It easily follows from the definition that
$$
\sum_{k=0}^{n-1} \eulerian{n}{k} = n! \qquad  \text{ and }  \qquad \eulerian{n}{k} = \eulerian{n}{n-k-1}.
$$
These two properties and the definition of $\eta_{d,m}$ imply that
$$
\sum_{m=0}^{\lfloor d/2\rfloor} \eta_{d,m} \cdot \frac{\eulerian{d+1}{m}}{(d+1)!} = \frac 1 {(d+1)!} \sum_{m=0}^{d} \eulerian{d+1}{m} = 1,
$$
i.e.\ the probabilities of types indeed sum up to $1$.
\end{remark}
\begin{proof}[Proof of Theorem~\ref{theo:sylvester_probab_types_conv_RW}]
If $(S_i)_{i=0}^n$ is a random walk in $\R^d$, $n\geq d$,  whose increments $\xi_1,\dots,\xi_n$ satisfy conditions $(\text{Ex})$ and $(\text{GP})$, then for all $\ell\in \{0,\ldots, d-1\}$, 
\begin{align}
\E [f_\ell(\cQ_{n,d})]
&=
\frac{2\cdot \ell!}{n!} \sum_{j=0}^{\infty}\stirling{n+1}{d-2j}  \stirlingsec{d-2j}{\ell+1},
\label{eq:E_F_k_C_n_main_theorem}
\\
\binom{n+1}{\ell+1} -  \E [f_\ell(\cQ_{n,d})]
&=
\frac{2\cdot \ell!}{n!} \sum_{j=1}^{\infty}\stirling{n+1}{d+2j}  \stirlingsec{d+2j}{\ell+1}.
\label{eq:E_F_k_C_n_main_theorem_complement}
\end{align}
Here, $\stirling{n}{k}$ and $\stirlingsec{n}{k}$ denote the Stirling numbers of the first and second kind, respectively; see, e.g., \cite[Section~6.1]{graham_knuth_patashnik_book}.
The first formula was derived in~\cite[Theorem~1.2]{KVZ17}, while the second one follows from the first one and the properties of Stirling numbers, see~\cite[Section~6.1, Eqn.~(6.3)]{kabluchko_marynych_lah_distr}. For $n=d+1$, \eqref{eq:E_F_k_C_n_main_theorem_complement} simplifies to
$$
\binom{d+2}{\ell+1} -  \E [f_\ell(\cQ_{d+1,d})]
=
\frac{2\cdot \ell!}{(d+1)!} \stirlingsec{d+2}{\ell+1},
\qquad
\ell\in \{0,\ldots, d-1\}.
$$
Plugging these values into Theorem~\ref{theo:sylvester_probab_types_trhough_expected_f_vectors} gives
\begin{align*}
p_{d,m}^{\text{convRW}}
&=
\frac 12 \eta_{d,m} \cdot \sum_{\ell=0}^{m} (-1)^{m+\ell} \binom{d-\ell+1}{d-m+1} \left(\binom{d+2}{\ell+1} - \E f_{\ell }(\cQ_{d+1,d})\right)
\\
&=
\eta_{d,m} \cdot \sum_{\ell=0}^{m} (-1)^{m+\ell} \binom{d-\ell+1}{d-m+1} \frac{\ell!}{(d+1)!}   \stirlingsec{d+2}{\ell+1}
=
\eta_{d,m} \cdot \frac{\eulerian{d+1}{m}}{(d+1)!},
\end{align*}
where in the last step we used the identity
\begin{equation}\label{eq:eulerian_numbers_identity}
\eulerian{n}{m} = \sum_{\ell=0}^{m} (-1)^{m+\ell} \binom{n-\ell}{n-m}  \ell! \stirlingsec{n+1}{\ell+1}
\end{equation}
whose proof will be given in Section~\ref{subsec:eulerian_numbers}, see~\eqref{eq:eulerian_numbers_frobenius_0}.
\end{proof}

Recently, a different proof of Theorem~\ref{theo:sylvester_probab_types_conv_RW}, using Gale duality,  has been given by Barysheva~\cite{Barysheva2025RandomConvexHulls}.

\begin{example}[Sylvester problem for random walks]
Taking $m=0$ in Theorem~\ref{theo:sylvester_probab_types_conv_RW} we recover a result of \cite{panzo_sylvester_random_walk}:
$$
p_{d,0}^{\text{convRW}}
=
\P[\cQ_{d+1,d} \text{ is a simplex}]
=
2 \frac{\eulerian{d+1}{0}}{(d+1)!} = \frac {2}{(d+1)!}.
$$
\end{example}
\begin{remark}
It is well-known (see, e.g., \cite{carlitz_etal_CLT_eulerian_numbers}, \cite{hwang_etal_asympt_eulerian_recurrences} or~\cite[Theorem~4.5]{gawronski_neuschel_euler_frobenius}) that the Eulerian numbers satisfy a central limit theorem of the following form: For all $t\in \R$,
$$
\lim_{n\to\infty} \frac{1}{n!} \sum_{k=0}^{m(n)} \eulerian{n}{k}  = \frac{1}{\sqrt {2\pi}} \int_{-\infty}^t \eee^{-s^2/2} \dint s,
\quad
\text{ where }
\quad
m(n) = \frac n2 + t \sqrt{\frac{n}{12}} + o(\sqrt n).
$$
As a consequence of Theorem~\ref{theo:sylvester_probab_types_conv_RW} and this CLT, one can easily deduce the following limit theorem for the type of $\cQ_{d+1,d}$ as $d\to\infty$.
\end{remark}
\begin{corollary}[A CLT for probabilities of types]\label{cor:probab_types_convex_hulls_RW_CLT}
For every $t>0$,
$$
\lim_{d\to\infty}\P\left[\cQ_{d+1,d} \text{ is of type } T_m^{d} \text{ with } m \geq \frac d2 - t \sqrt{\frac{d}{12}}\right] = \frac{2}{\sqrt{2\pi}} \int_{0}^t \eee^{-s^2/2} \dint s.
$$
\end{corollary}

\begin{remark}
Schmidt and Simion~\cite{schmidt_simion_geom_comb_eulerian_num} solved an analogue of the Youden problem for the uniform distribution. In probabilistic terms, their result states that  if $U_1,\ldots, U_n$ are i.i.d.\ uniform on $[0,1]$ with order statistics $U_{1:n} < \ldots < U_{n:n}$ , then
$$
\P\left[U_{i:n} \leq \frac{0 + U_1 + \ldots + U_n + 1}{n+2} \leq U_{i+1:n}\right] = \frac{\eulerian{n+1}{i}}{(n+1)!},
\qquad
i\in \{0,\ldots, n\}.
$$
Here, we put $U_{0:n}=0$ and $U_{n+1:n} := 1$. We were not able to establish a direct connection between their result and Theorem~\ref{theo:sylvester_probab_types_conv_RW}.
\end{remark}

\section{Probabilities of types for cones or  spherical polytopes}
\subsection{Probabilities of types in terms of the expected face numbers: spherical version}
Let $X_1,\ldots, X_{d+2}$ be (possibly dependent) random vectors in $\R^{d+1}$ which are a.s.\ in general linear position.  This means that any $d+1$ vectors from this collection are linearly independent with probability $1$.  The polyhedral cone spanned by $X_1,\ldots, X_{d+2}$ is denoted by
$$
C := \pos(X_1,\ldots, X_{d+2}) = \{\lambda_1 X_1 + \ldots + \lambda_{d+2} X_{d+2}: \lambda_1 \geq 0, \ldots, \lambda_{d+2}\geq 0\}.
$$
The intersection of $C$  with the unit sphere, denoted by $Q: = C\cap \bS^d$, is a $d$-dimensional spherical polytope with at most $d+2$ vertices. Note that it is possible that $Q=\bS^d$, which is an important difference compared to what we have seen before. By the general linear position assumption, all proper faces of $Q$ are spherical simplices. Let $f_j(Q) = f_{j+1}(C)$ be the number of $j$-dimensional faces of $Q$, for all $j\in \{0,\ldots, d\}$. It is also natural to put
$$
f_{-1} (Q) := f_0(C) =  \ind_{\{C \neq \R^d\}} =  \ind_{\{Q \neq \bS^d\}}.
$$
Now, the possible combinatorial types of $Q$ are defined as follows: we say that $Q$ is of type $T_{m}^d$ with $m\in \{-1,0,\ldots, \lfloor d/2\rfloor\}$ if $f_{d-1}(Q) = (m+1)(d+1-m)$.
For example, in the degenerate case when $Q = \bS^d$ we say that $Q$ is of type $T_{-1}^d$. Further, if $Q$ is a spherical simplex, we say that $Q$ is of type $T_{0}^d$. Otherwise, $Q$ has $d+2$ vertices and can be assigned to one of the types $T_1^{d}, \ldots, T_{\lfloor d/2 \rfloor}^d$.

We are ready to state a spherical (or conic) version of Theorem~\ref{theo:sylvester_probab_types_trhough_expected_f_vectors}.
\begin{theorem}[Probabilities of types and expected $f$-vectors: spherical version]\label{theo:sylvester_probab_types_conical_trhough_expected_f_vectors}
Let $X_1,\ldots, X_{d+2}$ be random vectors in $\R^{d+1}$ which are in general linear position with probability $1$. Consider the spherical polytope $Q:= \pos (X_1,\ldots, X_{d+2}) \cap \bS^d$. Then, for all $m\in \{-1,0,\ldots, \lfloor d/2\rfloor\}$ we have
$$
q_{d,m} := \P[Q \text{ is of type } T_{m}^d]
=
\frac 12 \eta_{d,m} \cdot \sum_{\ell=-1}^{m} (-1)^{m+\ell} \binom{d-\ell+1}{d-m+1} \left(\binom{d+2}{\ell+1} - \E f_{\ell }(Q)\right),
$$
where $\eta_{d,m}$ is the same as in~\eqref{eq:sylvester_types_eta_d_m_def} above, that is
$$
\eta_{d,m}
=
1 + \ind_{\{m\neq d/2\}}
=
\begin{cases}
2,
&\text{ if } m\neq d/2,
\\
1, &\text{ if } m=d/2.
\end{cases}
$$
\end{theorem}
\begin{proof}
Let us write  $q_m := q_{d,m}$ be the probability that $Q$ is of combinatorial type $T_m^d$, for $m\in \{-1,0,\ldots, \lfloor d/2\rfloor\}$.  By the formula of the total probability,
\begin{equation}\label{eq:probab_of_types_linear_eqs_conic}
\E f_j (Q) = \sum_{m = -1}^{\lfloor d/2 \rfloor} q_{m} f_j (T_m^d), \qquad j\in \{0, \ldots, d-1\}.
\end{equation}
Recall from Theorem~\ref{theo:facets_polys_d+2_vert} that for $m\in \{0,\ldots, \lfloor d/2\rfloor\}$,
\begin{equation}\label{eq:polytope_d+2_vert_f_vector_conic}
\binom{d+2}{j+1} - f_j (T_m^d) = \binom{d-m+1}{d-j+1} + \binom{m+1}{d-j+1},
\qquad
j\in \{0,\ldots, d-1\}.
\end{equation}
This continues to hold if $m=-1$ or $j=-1$ since $f_j(T_{-1}^d) = f_j(\bS^d) = 0$ for $-1\leq j\leq d-1$ and $f_{-1}(T_m^d) = 1$ for $m\geq 0$.
It follows that for all $j,m\in \{-1,0,\ldots, \lfloor d/2 \rfloor\}$,
$$
\binom{d+2}{j+1} - f_j (T_m^d) = \binom{d-m+1}{d-j+1} + \ind_{\{m = j = d/2\}} = \binom{d-m+1}{j-m} + \ind_{\{m = j = d/2\}}.
$$
Using $q_{-1} + q_0 + \ldots + q_{\lfloor d/2 \rfloor} = 1$ we rewrite~\eqref{eq:probab_of_types_linear_eqs_conic} in the form
\begin{align}
z_j
:&= \binom {d+2}{j+1} -  \E f_j(Q)
=
\sum_{m= -1}^{\lfloor d/2 \rfloor} \left(\binom {d+2}{j+1} - f_j(T_m^d)\right) q_{m}\notag
\\
&= \sum_{m= -1}^{j} \left(\binom{d-m+1}{j-m} + \ind_{\{m = j = d/2\}}\right) q_m,
\qquad
j\in \{-1, 0,\ldots, \lfloor d/2 \rfloor\}. \label{eq:probab_of_types_linear_eqs_complements_conic}
\end{align}

This is a system of  $\lfloor d/2\rfloor +2$ linear equations with the same number of unknowns.
Consider the vectors $z = (z_{-1}, z_0, \ldots, z_{\lfloor d/2 \rfloor})^\top$ and $q = (q_{-1}, q_0, \ldots, q_{\lfloor d/2 \rfloor})^\top$ and the lower-triangular matrix
$$
A := (a_{jm})_{j, m = -1,0, \ldots, \lfloor d/2 \rfloor} \text{ with }  a_{j m} = \binom{d-m+1}{j-m} + \ind_{\{j = m = d/2\}}.
$$
The system~\eqref{eq:probab_of_types_linear_eqs_complements_conic} takes the form $q = A q$, and its solution is $q = A^{-1} q$. The matrix $A$ is lower-triangular, that is $a_{jm} = 0$ for $m>j$. Its diagonal entries are given by $a_{jj} = 1 + \ind_{\{j = m = d/2\}}$.  Hence, the inverse matrix $A^{-1}$ indeed exists. We claim that it has the form
$$
B = (b_{m\ell})_{m, \ell = -1,0, \ldots, \lfloor d/2 \rfloor} \text{ with }  b_{m \ell} = (-1)^{m+\ell} \binom{d-\ell+1}{m-\ell} \cdot
\frac 12 \eta_{d,m}.
$$
The matrix $B$ is also lower-triangular, that is $b_{m\ell} = 0$ for $\ell>m$, and its diagonal entries are equal to $1$ except the last one. Showing that $AB$ is the identity matrix amounts to proving the identity $\sum_{m=-1}^{j} a_{jm} b_{m\ell} = \delta_{j\ell}$ for all $-1\leq \ell \leq j \leq \lfloor d/2\rfloor$. If $d$ is even, the term $\ind_{\{j = m = d/2\}}$ multiplies the only non-zero element in the last column of $A$ by two, which is compensated by the factor $\frac 12 \eta_{d,m}$ in the definition of $B$ which multiplies the last row of $B$ by $1/2$, if $d$ is even. So, it remains to show that
$$
\sum_{m=\ell}^j \binom{d-m+1}{j-m} (-1)^{m+\ell} \binom{d-\ell+1}{m-\ell}  = \delta_{j\ell}
$$
for all $-1\leq \ell \leq j \leq \lfloor d/2\rfloor$. This has already been done in~\eqref{eq:sylvester_types_faces_identity_binom_coeff}.
\end{proof}

\subsection{Wendel--Donoho--Tanner random cones}\label{subsec:wendel_donoho_tanner}
Let $X_1,\ldots, X_n$ be (possibly dependent) random vectors in $\R^{d+1}$, where $n\geq d+1$. We impose the following assumptions on their joint distribution:
\begin{itemize}
\item[($\pm$)]
The joint law of $(X_1,\ldots, X_n)$ is \textit{even} meaning that for every vector of signs $(\eps_1,\ldots, \eps_n)\in \{+1,-1\}^n$ the joint law of $(\eps_1 X_1,\ldots, \eps_nX_n)$ is the same as the joint law of $(X_1,\ldots, X_n)$.
\item[(GP)] The vectors $X_1,\ldots, X_n$ are in general linear position with probability $1$ meaning that for every $1\leq j_1 < \ldots < j_{d+1} \leq n$, the vectors $X_{j_1},\ldots, X_{j_{d+1}}$ are linearly independent with probability $1$.
\end{itemize}
Consider the cone $\cD_{n,d+1}=\pos(X_1,\dots,X_n)$ which is called Wendel's or Donoho--Tanner cone in the literature.  Under the above assumptions, Wendel~\cite{wendel}  (see also~\cite[Theorem~8.2.1]{schneider_weil_book}) proved that
\begin{align}
\P[\cD_{n,d+1} \neq \R^{d+1}] =  \frac 1  {2^{n-1}} \sum_{r=0}^{d} \binom{n-1}{r},
\qquad
\P[\cD_{n,d+1} = \R^{d+1}] = \frac 1  {2^{n-1}} \sum_{r=d+1}^{n-1} \binom{n-1}{r}.
\label{eq:wendel_formula_sylvester_types}
\end{align}

\citet[Theorem~1.6]{donoho_tanner1} computed the expected face numbers:
\begin{equation}\label{eq:donoho_tanner_expected_face_number}
\binom {n}{j} - \E f_j(\cD_{n,d+1}) = \frac{1}{2^{n-j-1}}\binom{n}{j}\sum_{r=0}^{n-d-2}\binom{n-j-1}{r},
\qquad
j\in \{0,\dots, d\}.
\end{equation}
We shall be interested in the case $n=d+2$. The next result describes the distribution of the type of the spherical polytope $\cD_{d+2,d+1} \cap \bS^d$.
\begin{theorem}[Probabilities of types for Wendel cones]
Let $X_1,\ldots,X_{d+2}$ be random vectors in $\R^{d+1}$ satisfying Assumptions~($\pm$) and (GP). Then, for all $m\in \{-1, 0,1,\ldots, \lfloor d/2 \rfloor\}$,
$$
q_{d,m}^{\text{Wend}} := \P[\cD_{d+2,d+1} \cap \bS^d \text{ is of type } T_m^{d}]
=
\frac{\eta_{d,m}}{2^{d+2}} \binom{d+2}{m+1}.
$$
\end{theorem}
\begin{proof}
Taking $n=d+2$ and $j=\ell+1$ in~\eqref{eq:donoho_tanner_expected_face_number}  gives
$$
\binom {d+2}{\ell+1} - \E f_{\ell}(\cD_{d+2,d+1}\cap \bS^d) = \frac{1}{2^{d-\ell}}\binom{d+2}{\ell+1},
\qquad
\ell\in \{-1,0, \dots, d-1\}.
$$
By Theorem~\ref{theo:sylvester_probab_types_conical_trhough_expected_f_vectors},
\begin{align*}
q_{d,m}^{\text{Wend}}
&=
\frac 12 \eta_{d,m} \cdot \sum_{\ell=-1}^{m} (-1)^{m+\ell} \binom{d-\ell+1}{d-m+1} \left(\binom{d+2}{\ell+1} - \E f_{\ell }(\cD_{d+2,d+1}\cap \bS^d)\right)
\\
&=
\frac 12 \eta_{d,m} \cdot \sum_{\ell=-1}^{m} (-1)^{m+\ell} \binom{d-\ell+1}{d-m+1} \frac{1}{2^{d-\ell}}\binom{d+2}{\ell+1}
\\
&=
\frac{\eta_{d,m}}{2^{d+2}} \binom{d+2}{m+1} \cdot \sum_{\ell=-1}^{m} \binom{m+1}{\ell+1} 2^{\ell+1} (-1)^{m-\ell}
\\
&=
\frac{\eta_{d,m}}{2^{d+2}} \binom{d+2}{m+1},
\end{align*}
where in the last step we used the binomial theorem.

Let us sketch a second, geometric proof of the theorem. We represent $\cD_{d+2,d+1}$ as the positive hull of the vector $X_{d+2}$ and the cone $B: = \pos (X_1,\ldots, X_{d+1})$. Note that $B$ is a full-dimensional cone in $\R^{d+1}$. Consider the linear hyperplanes
$$
H_i = \lin (X_1,\ldots, X_{i-1}, X_{i+1}, \ldots, X_{d+1}), \qquad i\in \{1,\ldots, d+1\}.
$$
Let $H_i^0$ and $H_i^1$ denote the closed half-spaces bounded by $H_i$ and let $B\subseteq H_i^0$.  The hyperplanes $H_1,\ldots, H_{d+1}$ dissect $\R^{d+1}$ into $2^{d+1}$ polyhedral cones of the form
$$
B_{\eps_1,\ldots, \eps_{d+1}} := H_1^{\eps_1} \cap \dots \cap H_{d+1}^{\eps_{d+1}},
\qquad
(\eps_1,\ldots, \eps_{d+1})\in \{0,1\}^{d+1}.
$$
Two of these cones are $B = B_{0,\ldots, 0}$ and $-B = B_{1,\ldots, 1}$. Now, since the events $X_{d+2} \in B_{\eps_1,\ldots, \eps_{d+1}}$ are equiprobable, disjoint, and one of these events occurs a.s., the probability of each such event is $1/2^{d+1}$.
Further,  it is possible to check that $\cD_{d+2,d+1}\cap \bS^d = \pos (B, \{X_{d+2}\}) \cap \bS^d$ is of type $T_{m}^d$ if and only if $X_{d+2}$ belongs to one of the cones $B_{\eps_1,\ldots, \eps_{d+1}}$ with $\eps_1 + \ldots  + \eps_{d+1} \in \{m, d-m\}$, for all $m\in \{-1,0,\ldots, \lfloor d/2\rfloor\}$, modulo zero events. (This is a conic version of the beneath/beyond characterization of $T_m^d$ in~\cite[p.~98]{gruenbaum_book}. For example, $\cD_{d+2,d+1}= \bS^d$ iff $X_{d+2}\in -B$ (case $m=-1$). Also,  $\cD_{d+2,d+1}\cap \bS^d$ is a spherical simplex iff $X_{d+2}\in B$ or $X_{d+2}$ belongs to one of the $d+1$ cones $B_{\eps_1,\ldots,\eps_{d+1}}$ with $\eps_1+\ldots+\eps_{d+1} = d$, which is the case $m=0$.)
The number of cones $B_{\eps_1,\ldots, \eps_{d+1}}$ with $\eps_1 + \ldots  + \eps_{d+1} = m$ (respectively, $d-m$) is $\binom{d+1}{m}$ (respectively, $\binom{d+1}{d-m} = \binom{d+1}{m+1}$).  Note that in the exceptional case when $m= d/2$ these collections of cones coincide, and the number of cones is $\binom{2m+1}{m} = \frac 12 \cdot \binom{2m+2}{m+1} = \frac 12 \cdot \binom{d+2}{m+1}$.   If $m\neq d/2$,  these two collections are disjoint and the total number of cones  is $\binom{d+1}{m} + \binom{d+1}{m+1} = \binom{d+2}{m+1}$. In both cases, the probability that $X_{d+2}$ is in any such cone is $1/2^{d+1}$ and we conclude that $\cD_{d+2,d+1} \cap \bS^d$ is of type $T_m^{d}$ with probability
$\frac{\eta_{d,m}}{2^{d+2}} \binom{d+2}{m+1}$.
\end{proof}

\begin{example}[Sylvester problem on a sphere]
The special cases $m=-1$ and $m=0$ allow for very simple proofs: Letting $B = \pos (X_1,\ldots, X_{d+1})$ we have $\P[X_{d+2} \in B] = \P[X_{d+2} \in -B] = 1/2^{d+1}$ for symmetry reasons (see the above proof) and hence
\begin{align*}
q_{d,-1}^{\text{Wend}}
&=
\P[\cD_{d+2,d+1}=\R^{d+1}]
=
\P[\pos(B, \{X_{d+2}\}) = \R^{d+1}] = \P[X_{d+2} \in -B] = \frac{1}{2^{d+1}},
\\
q_{d,0}^{\text{Wend}}
&=
\P[\cD_{d+2,d+1}\cap \bS^{d} \text{ is a spherical simplex}]
=
(d+2) \cdot  \P[X_{d+2} \in B] = \frac{d+2}{2^{d+1}}.
\end{align*}
The first formula is also an immediate consequence of Wendel's formula~\eqref{eq:wendel_formula_sylvester_types},
while the second one recovers a result of~\citet{maehara_martini_sylvester_sphere} (low-dimensional cases are treated in~\cite[Theorem~3.6]{maehara_martini_geometric_probab} and~\cite[Corollary~8.1]{maehara_martini_book}).
\end{example}

The central limit theorem for the binomial distribution $\Bin (d+2, \frac 12)$ yields the following
\begin{corollary}[A CLT for probabilities of types]
For every $t>0$,
$$
\lim_{d\to\infty}\P\left[\cD_{d+2,d} \text{ is of type } T_m^{d} \text{ with } m \geq \frac d2 - t \sqrt{\frac{d}{4}}\right] = \frac{2}{\sqrt{2\pi}} \int_{0}^t \eee^{-s^2/2} \dint s.
$$
\end{corollary}

\subsection{Positive hulls of random walks and bridges}\label{subsec:walk_bridges_positive_hull_refined_sylvester}
To compute probabilities of types for positive hulls of random walks, we shall rely on the results of~\cite{godland_kabluchko_positive_hulls_rand_walks}, where the corresponding expected $f$-vectors have been determined. As it turns out, in the conical setting there are two natural models: symmetric random walks and random bridges, related to reflection groups of types $B$ and $A$, respectively.  We recall the definitions from~\cite[Section~1.2]{godland_kabluchko_positive_hulls_rand_walks}.

To define random walks, let $X_1,\ldots,X_n$ be (possibly dependent) random vectors in $\R^{d+1}$. Consider their partial sums $S_1,\ldots,S_n$ defined by $S_k:=X_1+\ldots+X_k$ for $k\in\{1,\ldots,n\}$.  We say that $S_1,\ldots,S_n$ is a \textit{(symmetric) random walk} in $\R^{d+1}$ starting at $S_0:=0$ if the following conditions are satisfied:
\begin{enumerate}
\item[($\pm$Ex)]\emph{Symmetric exchangeability}: For every permutation $\sigma= (\sigma(1),\ldots, \sigma(n))$ of $\{1,\ldots, n\}$ and every vector of signs $\eps=(\eps_1,\ldots,\eps_n)\in\{\pm 1\}^n$, we have that
\begin{align*}
(\eps_1X_{\sigma(1)},\ldots,\eps_nX_{\sigma(n)})\eqdistr (X_1,\ldots,X_n).
\end{align*}
\item[(GP)] \emph{General position}:  $n\ge d+1$ and, for every $1\le i_1<\ldots<i_{d+1}\le n$, the vectors $S_{i_1},\ldots,S_{i_{d+1}}$ are linearly independent with probability $1$.
\end{enumerate}
The positive hull of the symmetric random walk is then defined as
\begin{align*}
\cC_{n,d+1}^B:=\pos\{S_1,\ldots,S_n\}
\end{align*}

To define random bridges, let $\widetilde X_1,\ldots,\widetilde X_n$ be (possibly dependent) random vectors in $\R^{d+1}$ with partial sums $\widetilde S_k:= \widetilde X_1+ \ldots + \widetilde X_k$. We say that $\widetilde S_1,\ldots,\widetilde S_n$ is a \textit{random bridge} in $\R^{d+1}$ if the following conditions are satisfied:
\begin{enumerate}
\item[(Ex)]\emph{Exchangeability}: For every permutation $\sigma= (\sigma(1),\ldots, \sigma(n))$ of $\{1,\ldots, n\}$, we have that
\begin{align*}
(\widetilde X_{\sigma(1)},\ldots,\widetilde X_{\sigma(n)})\eqdistr (\widetilde X_1,\ldots,\widetilde X_n).
\end{align*}
\item[(Br)]  \emph{Bridge property}: With probability $1$, it holds that $\widetilde S_n=\widetilde X_1+\ldots+\widetilde X_n=0$.
\item[(GP')] \emph{General position}: $n\ge d+2$ and, for every $1\le i_1<\ldots<i_{d+1}\le n-1$, the vectors $\widetilde S_{i_1},\ldots,\widetilde S_{i_{d+1}}$ are linearly independent with probability $1$.
\end{enumerate}
The bridge starts and ends at $0=\widetilde S_0=\widetilde S_n$.
The positive hull of this random bridge is then defined as
\begin{align*}
\cC_{n,d+1}^A:=\pos\{\widetilde S_1,\ldots,\widetilde S_{n}\} = \pos\{\widetilde S_1,\ldots,\widetilde S_{n-1}\}.
\end{align*}
We shall be interested in the case $n=d+2$ (for walks) and $n=d+3$ (for bridges), so that  the number of non-zero vectors generating the cone is always $d+2$.
\begin{theorem}[Probabilities of types for positive hulls of random walks and bridges]
\label{theo:sylvester_probab_types_conical_RW_bridges}
Under the above assumptions, for all $m\in \{-1,0,\ldots, \lfloor d/2 \rfloor\}$,
\begin{align*}
q_{d,m}^{\text{posBr}} := \P[\cC_{d+3,d+1}^A\cap \bS^d \text{ is of type } T_m^{d}]
&=
\eta_{d,m} \cdot \frac{\eulerian{d+3}{m+1}}{(d+3)!},
\\
q_{d,m}^{\text{posRW}}  := \P[\cC_{d+2,d+1}^B\cap \bS^d \text{ is of type } T_m^{d}]
&=
\eta_{d,m} \cdot \frac{\eulerianb{d+2}{m+1}}{2^{d+2}(d+2)!},
\end{align*}
where $\eulerianb{n}{k}$ are the $B$-analogues of the Eulerian numbers $\eulerian{n}{k}$; see Section~\ref{subsec:eulerian_numbers} below for more information and references.
\end{theorem}

\begin{proof}
Let $j\in\{0,\ldots,d\}$ be given. It was shown in~\cite[Corollary~2.12]{godland_kabluchko_positive_hulls_rand_walks} (where the random walks live in $\R^d$ rather than in $\R^{d+1}$) that
\begin{align*}
\E f_j(\cC_{n,d+1}^A)
&	= 2\cdot \frac{(j+1)!}{n!}\sum_{r=0}^\infty \stirling{n}{d+1-2r}\stirlingsec{d+1-2r}{j+1},\\
\E f_j(\cC_{n,d+1}^B)
&	= 2\cdot \frac{2^j j!}{2^{n} n!}\sum_{r=0}^\infty \stirlingb{n}{d-2r}\stirlingsecb{d-2r}{j}.
\end{align*}
Here, $\stirlingb{n}{k}$ and $\stirlingsecb{n}{k}$ denote the $B$-analogues of the Stirling numbers of the first and second kind; see, e.g., \cite{Schlafli_angle_sums,godland_kabluchko_positive_hulls_rand_walks}.
We now use the identities $\sum_{i\text{ even}} \stirling{n}{i} \stirlingsec{i}{m} = \sum_{i\text{ odd}} \stirling{n}{i} \stirlingsec{i}{m} = \frac 12 \frac{(n-1)!}{(m-1)!} \binom nm$ and $\sum_{i\text{ even}} \stirlingb{n}{i} \stirlingsecb{i}{m} = \sum_{i\text{ odd}} \stirlingb{n}{i} \stirlingsecb{i}{m} = \frac{1}{2} \frac{2^n n!}{2^m m!} \binom{n}{m}$, both valid for $n>m$; see, e.g., \cite[Corollary 3.13, Equation (2.23)]{Schlafli_angle_sums}.  These identities imply $\sum_{r=0}^\infty \stirling{n}{d+1-2r}\stirlingsec{d+1-2r}{j+1} = \frac 12 \frac{(n-1)!}{j!} \binom n{j+1} - \sum_{r=1}^\infty \stirling{n}{d+1+2r}\stirlingsec{d+1+2r}{j+1}$ and similarly in the $B$-case. It follows that
\begin{align*}
\binom{n-1}{j} - \E f_j(\cC_{n,d+1}^A)
&	= 2\cdot \frac{(j+1)!}{n!}\sum_{r=1}^\infty \stirling{n}{d+1+2r}\stirlingsec{d+1+2r}{j+1},\\
\binom nj - \E f_j(\cC_{n,d+1}^B)
&	= 2\cdot \frac{2^j j!}{2^{n} n!}\sum_{r=1}^\infty \stirlingb{n}{d+2r}\stirlingsecb{d+2r}{j}.
\end{align*}
Now we take $n=d+2$ (for walks) and $n=d+3$ (for bridges). Also we intersect the positive hulls with $\bS^d$ and put $j= \ell+1$  with $\ell \in \{-1,0,\ldots, d-1\}$. Then, these formulas simplify to
\begin{align}
\binom{d+2}{\ell+1} - \E f_\ell(\cC_{d+3,d+1}^A \cap \bS^d)
&	
=
2\cdot \frac{(\ell+2)!}{(d+3)!} \stirlingsec{d+3}{\ell+2},  \label{eq:sylvester_conic_bridge_proof_missing_f_vect}
\\
\binom {d+2}{\ell+1} - \E f_\ell(\cC_{d+2,d+1}^B\cap \bS^d)
&	= 2\cdot \frac{2^{\ell+1} (\ell+1)!}{2^{d+2} (d+2)!}\stirlingsecb{d+2}{\ell+1}.  \label{eq:sylvester_conic_walk_proof_missing_f_vect}
\end{align}
For bridges, plugging~\eqref{eq:sylvester_conic_bridge_proof_missing_f_vect} into Theorem~\ref{theo:sylvester_probab_types_conical_trhough_expected_f_vectors} gives
\begin{align*}
q_{d,m}^{\text{posBr}}
&=
\frac 12 \eta_{d,m} \cdot \sum_{\ell=-1}^{m} (-1)^{m+\ell} \binom{d-\ell+1}{d-m+1} \left(\binom{d+2}{\ell+1} - \E f_{\ell}(\cC_{d+3,d+1}^A \cap \bS^d)\right)
\\
&=
\eta_{d,m} \cdot \sum_{\ell=-1}^{m} (-1)^{m+\ell} \binom{d-\ell+1}{d-m+1} \frac{(\ell+2)!}{(d+3)!} \stirlingsec{d+3}{\ell+2}
\\
&=
\eta_{d,m} \cdot \frac{\eulerian{d+3}{d-m+1}}{(d+3)!}
=
\eta_{d,m} \cdot \frac{\eulerian{d+3}{m+1}}{(d+3)!},
\end{align*}
where we used the classical identity
$$
\eulerian {n}{k} = \eulerian {n}{n-k-1} =  \sum_{j=1}^{n-k} (-1)^{n-j-k} \binom{n-j}{k} j!\stirlingsec{n}{j};
$$
see Section~\ref{subsec:eulerian_numbers}, Equation~\eqref{eq:eulerian_numbers_frobenius_1}, for its proof.
For walks, plugging~\eqref{eq:sylvester_conic_walk_proof_missing_f_vect} into Theorem~\ref{theo:sylvester_probab_types_conical_trhough_expected_f_vectors} gives
\begin{align*}
q_{d,m}^{\text{posRW}}
&=
\frac 12 \eta_{d,m} \cdot \sum_{\ell=-1}^{m} (-1)^{m+\ell} \binom{d-\ell+1}{d-m+1} \left(\binom{d+2}{\ell+1} - \E f_\ell(\cC_{d+2,d+1}^B\cap \bS^d)\right)
\\
&=
\eta_{d,m} \cdot \sum_{\ell=-1}^{m} (-1)^{m+\ell} \binom{d-\ell+1}{d-m+1} \frac{2^{\ell+1} (\ell+1)!}{2^{d+2} (d+2)!}\stirlingsecb{d+2}{\ell+1}
\\
&=
\eta_{d,m} \cdot \frac{\eulerianb{d+2}{d-m+1}}{2^{d+2}(d+2)!} = \eta_{d,m} \cdot \frac{\eulerianb{d+2}{m+1}}{2^{d+2}(d+2)!},
\end{align*}
where in the last line we used the identity
$$
B\eulerian{n}{k} = B\eulerian{n}{n-k}  =  \sum_{j=0}^{n-k} (-1)^{n-j-k} \binom{n-j}{k} 2^j j!\stirlingsecb{n}{j},
$$
whose proof will be given in Section~\ref{subsec:eulerian_numbers}; see~\eqref{eq:eulerian_B_symmetry} and ~\eqref{eq:eulerian_numbers_frobenius_type_B}.
\end{proof}

\begin{remark}
The type $B$ Eulerian numbers satisfy a central limit theorem with the same centering and normalization as for the usual Eulerian numbers; see~\cite[Propositions~4.6, 4.7 and Theorem 6.2]{kahle_stump_counting_inversions_descents}. More precisely, for all $t\in \R$,
$$
\lim_{n\to\infty} \frac{1}{2^n n!} \sum_{k=0}^{m(n)} B\eulerian{n}{k}  = \frac{1}{\sqrt {2\pi}} \int_{-\infty}^t \eee^{-s^2/2} \dint s,
\quad
\text{ where }
\quad
m(n) = \frac n2 + t \sqrt{\frac{n}{12}} + o(\sqrt n).
$$
Theorem~\ref{theo:sylvester_probab_types_conical_RW_bridges} (together with Slutsky's lemma) implies that Corollary~\ref{cor:probab_types_convex_hulls_RW_CLT} holds with $\cQ_{d+1,d}$ replaced by $\cC_{d+3,d+1}^A$ or $\cC_{d+2,d+1}^B$.
\end{remark}


\section{Auxiliary results}
\subsection{Identities for the normal distribution function}\label{subsec:Phi_identity_integral_proof}
In this section we prove the following identity which was needed in Section~\ref{subsec:sylvester_refined_gaussian}:
\begin{equation}\label{eq:int_gauss_density_gauss_distr_funct_complex_arg_repeat}
\int_{-\infty}^{+\infty} \eee^{-x^2/2} \Phi^{n}\left(\frac{\ii x}{\sqrt{n}}\right) \dint x = \int_{-\infty}^{+\infty} \eee^{-x^2/2} \Phi^{n}\left(-\frac{\ii x}{\sqrt{n}}\right) \dint x = 0,\qquad n\in\{2,3,\ldots\}.
\end{equation}
Here, $\Phi(z)$ is the standard normal distribution function; see~\eqref{eq:Phi_def}. We begin with a standard
\begin{lemma}\label{lem:integral_analytic_function_cauchy_thm}
Let $f(w)$ be a function which is continuous on the upper half-plane $\{\Im w \geq 0\}$ and analytic on $\{\Im w>0\}$. Suppose that $f(w) = O(|w|^{-2})$ as $|w|\to+\infty$, $\Im w \geq 0$. Then,
\begin{equation}\label{eq:lem:integral_analytic_function_cauchy_thm}
\int_{-\infty}^{+\infty} f(x) \dint x = 0.
\end{equation}
\end{lemma}
\begin{proof}
By Cauchy's theorem, we can replace the integration path in~\eqref{eq:lem:integral_analytic_function_cauchy_thm} by the following one: $(-\infty, -R)\cup \gamma_R \cup (+R,\infty)$, where $R>1$ is arbitrary and $\gamma_R$ is the clockwise oriented half-circle $\{R \eee^{- i t}: t\in [-\pi, 0]\}$ contained in the upper half-plane. Using that $|f(w)| \leq C |w|^{-2}$ for all $|w|\geq 1$, we find that
$$
\max\left\{\left|\int_{-\infty}^R f(x) \dint x\right|, \left|\int_{R}^{+\infty} f(x) \dint x\right|\right\} \leq  C \int_R^{+\infty} |x|^{-2} \dint x \leq C R^{-1}
$$
and
$$
\left|\int_{\gamma_R} f(x)\dint x \right| \leq \pi R \cdot C R^{-2} \leq \pi C R^{-1}.
$$
These estimates hold for all $R>1$.  Letting $R\to\infty$ completes the proof of~\eqref{eq:lem:integral_analytic_function_cauchy_thm}.
\end{proof}

The next identity for $\Phi(x)$ can be found in~\cite[Remark~1.4 and Proposition~1.5]{kabluchko_zaporozhets_absorption}. We give here an independent proof.
\begin{lemma}\label{lem:Phi_power_n_integral_0}
For all $n\in \{2,3,\ldots\}$ and all $\ell\in \{0,\dots, n-2\}$ we have
\begin{equation}\label{eq:Phi_integral_identity}
\int_{-\infty}^{+\infty} \left(\eee^{-x^2/2} \Phi\left(\ii x\right)\right)^n x^\ell \dint x = 0.
\end{equation}
\end{lemma}

\begin{proof}
We shall apply Lemma~\ref{lem:integral_analytic_function_cauchy_thm} to the function $f(x) := (\eee^{-x^2/2} \Phi(\ii x))^n x^\ell$.  To this end, we need to verify that $f(w) = O(|w|^{-2})$ as $|w|\to+\infty$, $\Im w \geq 0$.  Fix $\eps>0$. It is known  that the function $\Phi(z)$ satisfies
$$
\Phi(z) \sim - \frac {1}{\sqrt{2\pi} z} \eee^{-z^2/2},
\quad
\text{ as $|z|\to \infty$  in the domain $|\arg z| >\frac \pi 4 + \eps$};
$$
see~\cite[Eq.~7.1.23 on p.~298]{abramowitz_stegun}, \cite[\S~7.12.1]{NIST:DLMF} or~\cite[Exercise~3.11 on pp.~95--96]{bleistein_handelsman_book}.
A full proof can be found, for example, in~\cite[Lemma 3.10]{kabluchko_klimovsky_GREM}.
Writing $z= ix$ we obtain
\begin{equation}\label{eq:Phi_auxiliary_asymptotics}
f(x) = (\eee^{-x^2/2} \Phi(\ii x))^n x^\ell \sim  \left(\frac {\ii}{\sqrt{2\pi} x}\right)^n x^\ell = O(|x|^{-(n-\ell)}),  \quad  \text{ as $|x|\to\infty, \; \Im x \geq 0$}
\end{equation}
Since $n-\ell\geq 2$, we conclude that $f(x) = O(|x|^{-2})$ as $|x|\to\infty$, $\Im x \geq 0$.    It remains to apply Lemma~\ref{lem:integral_analytic_function_cauchy_thm}.
\end{proof}

\begin{proof}[Proof of~\eqref{eq:int_gauss_density_gauss_distr_funct_complex_arg_repeat}]
Take $\ell = 0$ in Lemma~\ref{lem:Phi_power_n_integral_0} and make the substitution $w= x/\sqrt n$ or $w= -x/\sqrt n$.
\end{proof}
\begin{remark}
As an application of~\eqref{eq:int_gauss_density_gauss_distr_funct_complex_arg_repeat}  let us verify that the probabilities in Theorem~\ref{theo:sylvester_probab_types_gauss} sum up to $1$. Recalling the definition of $\eta_{d,m}$  from~\eqref{eq:sylvester_types_eta_d_m_def}, this amounts to showing that
\begin{equation}\label{eq:sylvester_refined_gauss_probab_sum_up_to_1}
\sum_{m=0}^{d}  \binom{d+2}{m+1} \frac 1 {\sqrt{2\pi}} \int_{-\infty}^{+\infty} \eee^{-x^2/2} \Phi^{m+1}\left(\frac{\ii x}{\sqrt{d+2}}\right) \Phi^{d+1-m}\left(-\frac{\ii x}{\sqrt{d+2}}\right) \dint x = 1.
\end{equation}
\end{remark}
\begin{proof}
By~\eqref{eq:int_gauss_density_gauss_distr_funct_complex_arg_repeat}, the left-hand side of~\eqref{eq:sylvester_refined_gauss_probab_sum_up_to_1} does not change if we add two more terms corresponding to $m=-1$ and $m=d+1$.  With this in mind,
\begin{align*}
\text{LHS \eqref{eq:sylvester_refined_gauss_probab_sum_up_to_1}}
&=
\sum_{m=-1}^{d+1}  \binom{d+2}{m+1} \frac 1 {\sqrt{2\pi}} \int_{-\infty}^{+\infty} \eee^{-x^2/2} \Phi^{m+1}\left(\frac{\ii x}{\sqrt{d+2}}\right) \Phi^{d+1-m}\left(-\frac{\ii x}{\sqrt{d+2}}\right) \dint x
\\
&=
\frac 1 {\sqrt{2\pi}} \int_{-\infty}^{+\infty} \eee^{-x^2/2} \left(\Phi\left(\frac{\ii x}{\sqrt{d+2}}\right)+ \Phi\left(-\frac{\ii x}{\sqrt{d+2}}\right)\right)^{d+2} \dint x
\\
&=
\frac 1 {\sqrt{2\pi}} \int_{-\infty}^{+\infty} \eee^{-x^2/2} \dint x = 1.
\end{align*}
Note that we used that $\Phi(z) + \Phi(-z) = 1$.
\end{proof}

\subsection{Eulerian numbers}\label{subsec:eulerian_numbers}
In this section, we collect the necessary background on the Eulerian numbers $\eulerian{n}{k}$ and their $B$-analogues $B\eulerian{n}{k}$.  It is convenient to view these two triangular arrays of numbers as special cases of the Euler--Frobenius numbers $A_{n,k,\rho}$, which depend on  an additional parameter $\rho \in \R$.
 For a review on these numbers we refer to the paper of Janson~\cite{janson_euler_frobenius_rounding} (see also~\cite{gawronski_neuschel_euler_frobenius}).  It is known~\cite{janson_euler_frobenius_rounding} that
\begin{equation}\label{eq:euler_frobenius_P_n_rho_def}
\sum_{j=0}^\infty (j+\rho)^n x^j = \frac{P_{n,\rho}(x)}{(1-x)^{n+1}}, \qquad |x|<1,
\end{equation}
where $P_{n,\rho}(x)$ is a polynomial of degree $n$ in $x$, for all $n\in \N_0$. The coefficients of this polynomial are called \emph{Euler-Frobenius numbers}:
$$
P_{n,\rho}(x) = \sum_{k=0}^n A_{n,k, \rho} x^k.
$$
For $\rho=1$ and $\rho = 0$, one recovers the Eulerian numbers: $\eulerian{n}{k}  = A_{n,k,1}= A_{n, k+1,0}$; see the books of~\citet{graham_knuth_patashnik_book}, \citet{mezo_book} and~\citet{petersen_book_eulerian_numbers} for their properties. The $B$-Eulerian numbers are related to the case $\rho=1/2$ and given by $B \eulerian{n}{k} = 2^n A_{n,k,1/2}$. These numbers count signed permutations of $\{1,\ldots, n\}$ with $k$ ascents (see, e.g.~\cite[Section~13.1]{petersen_book_eulerian_numbers} or~\cite[p.~278]{chow_gessel}) and appear in~\cite{bjoerner_brenti_book_comb_coxeter_groups,brenti_eulerian_coxeter_groups,chakerian_cube_slices,chow_gessel,dey_eulerian_CLT_Carlitz_for_types_B_D,gawronski_neuschel_euler_frobenius,macmahon_divisors_numbers,simion_convex_poly_enum}; see the paper of Janson~\cite{janson_euler_frobenius_rounding}, the book by Petersen~\cite{petersen_book_eulerian_numbers} and entry~A060187 in~\cite{sloane} for more information. It is known that $A_{n,k,\rho} = A_{n,n-k, 1-\rho}$; see~\cite[(A.15)]{janson_euler_frobenius_rounding}. Consequently,
\begin{equation}\label{eq:eulerian_B_symmetry}
B\eulerian{n}{k} = 2^n A_{n,k,1/2} = 2^n A_{n,n-k,1/2}  =  B\eulerian{n}{n-k},
\qquad
\eulerian{n}{k}   =  \eulerian{n}{n-k-1}.
\end{equation}

A classical identity due to Frobenius expresses $\eulerian{n}{k}$ through the Stirling numbers $\stirlingsec{n}{k}$. We need a generalization of this identity expressing $A_{n,k,\rho}$ through  the $r$-Stirling numbers $\stirlingsec{n}{k}_r$ (with $r=1-\rho$); see~\cite{nyul:2015} for an overview of various definitions and properties of these numbers. Here we shall only need the exponential generating function
\begin{equation}\label{eq:gen_fct_r-stirling2}
\sum_{n=k}^\infty \stirlingsec{n}{k}_r \frac{y^n}{n!} = \frac{\eee^{ry}}{k!} \left(\eee^y-1\right)^k,
\quad
k\in \N_0.
\end{equation}
The usual (partition) Stirling numbers and their $B$-analogues can be recovered via $\stirlingsec{n}{j} = \stirlingsec{n}{j}_0$ and $\stirlingsecb nj = 2^{n-j} \stirlingsec{n}{j}_{1/2}$.

\begin{lemma}[Generalized Frobenius identity]
For all $n\in \N_0$ and $\rho\in \R$ we have
\begin{equation}\label{eq:euler_frobenius_polynomials_through_stirlingsec}
P_{n,\rho}(x) = \sum_{j=0}^n  (x-1)^{n-j} \stirlingsec{n}{j}_{1-\rho} j!.
\end{equation}
Consequently, expanding $(x-1)^{n-j}$ and comparing the coefficients of $x^k$ we get
\begin{equation}\label{eq:euler_frobenius_A_n_k_rho_through_r_stirlingsec}
A_{n,k,\rho} = \sum_{j=0}^n (-1)^{n+j+k} \binom {n-j}{k} \stirlingsec{n}{j}_{1 - \rho} j!,
\qquad
k\in \{0,\ldots, n\}.
\end{equation}
\end{lemma}
\begin{proof}
It would be possible to deduce this formula from~\cite[Theorem~2.1]{srivastava_eta_euler_frobenius}, but we find it easier to give an independent proof.   A consequence of~\eqref{eq:euler_frobenius_P_n_rho_def} is  the exponential generating function
$$
\sum_{n=0}^\infty \frac{P_{n,\rho}(x)}{(1-x)^{n+1}} \frac{z^n}{n!} = \sum_{j=0}^\infty \sum_{n=0}^\infty \frac{z^n}{n!}(j+\rho)^n x^j = \sum_{j=0}^\infty \eee^{(j+\rho) z} x^j = \frac{\eee^{\rho z}}{1 - x \eee^z}.
$$
Expanding the geometric series and using~\eqref{eq:gen_fct_r-stirling2} with $r= 1-\rho$ and $y= -z$ we get
\begin{align*}
\frac{\eee^{\rho z}}{1 - x \eee^z}
&=
\frac{\eee^{(\rho-1) z}}{1-x} \cdot \frac{1}{1- \frac{1-\eee^{-z}}{1-x}}
=
\eee^{(\rho-1) z} \sum_{j=0}^\infty \frac{(1-\eee^{-z})^j}{(1-x)^{j+1}}
\\
&=
\sum_{j=0}^\infty \frac {(-1)^j j! } {(1-x)^{j+1}} \frac{\eee^{(1-\rho)(-z)} (\eee^{-z}-1)^j}{j!}
=
\sum_{j=0}^\infty \frac {(-1)^j j!} {(1-x)^{j+1}}  \sum_{n=j}^\infty \stirlingsec{n}{j}_{1-\rho} \frac{(-z)^n}{n!}.
\end{align*}
Comparing the coefficients of $z^n/n!$ gives the claimed formula~\eqref{eq:euler_frobenius_polynomials_through_stirlingsec}.
\end{proof}

Taking $\rho = 1, 0, 1/2$ in~\eqref{eq:euler_frobenius_A_n_k_rho_through_r_stirlingsec} and using the formulas $\stirlingsec{n}{j}_0 = \stirlingsec{n}{j}$, $\stirlingsec{n}{j}_1 =\stirlingsec{n+1}{j+1}$ and $\stirlingsecb nj = 2^{n-j} \stirlingsec{n}{j}_{1/2}$, we obtain three formulas we relied on:
\begin{align}
\eulerian{n}{k}
&=
A_{n,k,1} =  \sum_{j=0}^n (-1)^{n+j+k} \binom {n-j}{k} \stirlingsec{n}{j} j!,
\label{eq:eulerian_numbers_frobenius_1}
\\
\eulerian{n}{k}
&=
A_{n,k+1,0}
= \sum_{j=0}^n (-1)^{n+j+k+1} \binom {n-j}{k+1} \stirlingsec{n+1}{j+1} j!,
\label{eq:eulerian_numbers_frobenius_0}
\\
B\eulerian{n}{k}
&=
2^n A_{n,k,1/2} =  \sum_{j=0}^n (-1)^{n+j+k} \binom {n-j}{k} B\stirlingsec{n}{j} 2^j j!.
\label{eq:eulerian_numbers_frobenius_type_B}
\end{align}
The first formula is classical; see~\cite[Eqn.~(6.40)]{graham_knuth_patashnik_book} or~\cite[p.~152]{mezo_book} and  expresses the fact that the Eulerian numbers  are the entries of the $h$-vector of the type $A_{n-1}$ Coxeter complex (or the dual of a permutohedron); see~\cite[Theorems~5.3 and~11.3]{petersen_book_eulerian_numbers} or~\cite[p. 163]{simion_convex_poly_enum}. The last formula is a similar claim for the type $B_n$ Coxeter complex.


\section*{Acknowledgement}
ZK has been supported by the German Research Foundation under Germany's Excellence Strategy  EXC 2044/2 -- 390685587, \textit{Mathematics M\"unster: Dynamics - Geometry - Structure} and by the DFG priority program SPP 2265 \textit{Random Geometric Systems}.

\section*{Data Availability}
No datasets were generated or analyzed during the current study.


\begin{thebibliography}{74}
\providecommand{\natexlab}[1]{#1}
\providecommand{\url}[1]{\texttt{#1}}
\expandafter\ifx\csname urlstyle\endcsname\relax
  \providecommand{\doi}[1]{doi: #1}\else
  \providecommand{\doi}{doi: \begingroup \urlstyle{rm}\Url}\fi

\bibitem[Abramowitz and Stegun(1964)]{abramowitz_stegun}
M.~Abramowitz and I.~Stegun.
\newblock \emph{Handbook of mathematical functions with formulas, graphs, and
  mathematical tables}, volume~55 of \emph{National Bureau of Standards Applied
  Mathematics Series}.
\newblock U.S.~Government Printing Office, Washington, 1964.

\bibitem[Affentranger and Schneider(1992)]{affentranger_schneider_random_proj}
F.~Affentranger and R.~Schneider.
\newblock Random projections of regular simplices.
\newblock \emph{{Discrete Comput. Geom.}}, 7\penalty0 (1):\penalty0 219--226,
  1992.

\bibitem[B\'ar\'any(2008)]{BaranyBullSurvey}
I.~B\'ar\'any.
\newblock Random points and lattice points in convex bodies.
\newblock \emph{Bull. Amer. Math. Soc. (N.S.)}, 45\penalty0 (3):\penalty0
  339--365, 2008.
\newblock \doi{10.1090/S0273-0979-08-01210-X}.

\bibitem[Barvinok(2002)]{barvinok_book}
A.~Barvinok.
\newblock \emph{A course in convexity}, volume~54 of \emph{Graduate Studies in
  Mathematics}.
\newblock American Mathematical Society, Providence, RI, 2002.
\newblock \doi{10.1090/gsm/054}.
\newblock URL \url{https://doi.org/10.1090/gsm/054}.

\bibitem[Barysheva(2025)]{Barysheva2025RandomConvexHulls}
K.~Barysheva.
\newblock Random convex hulls.
\newblock Diploma thesis, Saint Petersburg State University, 2025.
\newblock Advisor: D. Zaporozhets.

\bibitem[{Baryshnikov} and {Vitale}(1994)]{baryshnikov_vitale}
Y.~M. {Baryshnikov} and R.~A. {Vitale}.
\newblock Regular simplices and {G}aussian samples.
\newblock \emph{{Discrete Comput. Geom.}}, 11\penalty0 (2):\penalty0 141--147,
  1994.
\newblock \doi{10.1007/BF02574000}.

\bibitem[Bj\"{o}rner and
  Brenti(2005)]{bjoerner_brenti_book_comb_coxeter_groups}
A.~Bj\"{o}rner and F.~Brenti.
\newblock \emph{Combinatorics of {C}oxeter groups}, volume 231 of
  \emph{Graduate Texts in Mathematics}.
\newblock Springer, New York, 2005.

\bibitem[Bleistein and Handelsman(1986)]{bleistein_handelsman_book}
N.~Bleistein and R.~A. Handelsman.
\newblock \emph{Asymptotic expansions of integrals}.
\newblock Dover Publications, Inc., New York, second edition, 1986.

\bibitem[Bokowski et~al.(1992)Bokowski, Richter-Gebert, and
  Schindler]{BokowskiRichterGebertSchindler1992OrderTypes}
J.~Bokowski, J.~Richter-Gebert, and W.~Schindler.
\newblock On the distribution of order types.
\newblock \emph{Computational Geometry: Theory and Applications}, 1:\penalty0
  127--142, 1992.
\newblock URL \url{https://doi.org/10.1016/0925-7721(92)90012-H}.

\bibitem[Brenti(1994)]{brenti_eulerian_coxeter_groups}
F.~Brenti.
\newblock {$q$}-{E}ulerian polynomials arising from {C}oxeter groups.
\newblock \emph{European J. Combin.}, 15\penalty0 (5):\penalty0 417--441, 1994.
\newblock \doi{10.1006/eujc.1994.1046}.
\newblock URL \url{https://doi.org/10.1006/eujc.1994.1046}.

\bibitem[Calka(2019)]{CalkaSurvey}
P.~Calka.
\newblock Some classical problems in random geometry.
\newblock In \emph{Stochastic geometry}, volume 2237 of \emph{Lecture Notes in
  Math.}, pages 1--43. Springer, Cham, 2019.

\bibitem[Carlitz et~al.(1972)Carlitz, Kurtz, Scoville, and
  Stackelberg]{carlitz_etal_CLT_eulerian_numbers}
L.~Carlitz, D.~C. Kurtz, R.~Scoville, and O.~P. Stackelberg.
\newblock Asymptotic properties of {E}ulerian numbers.
\newblock \emph{Z. Wahrscheinlichkeitstheorie und Verw. Gebiete}, 23:\penalty0
  47--54, 1972.
\newblock \doi{10.1007/BF00536689}.
\newblock URL \url{https://doi.org/10.1007/BF00536689}.

\bibitem[Chakerian and Logothetti(1991)]{chakerian_cube_slices}
D.~Chakerian and D.~Logothetti.
\newblock Cube slices, pictorial triangles, and probability.
\newblock \emph{Math. Mag.}, 64\penalty0 (4):\penalty0 219--241, 1991.
\newblock \doi{10.2307/2690829}.
\newblock URL \url{https://doi.org/10.2307/2690829}.

\bibitem[Chan et~al.(2025)Chan, Kalai, Narayanan, Ter-Saakov, and
  White]{ChanKalaiNarayananTerSaakovWhite2025UnimodalityRadon}
Swee~Hong Chan, G.~Kalai, B.~Narayanan, N.~Ter-Saakov, and M.~White.
\newblock Unimodality for {R}adon partitions of random vectors.
\newblock Preprint at \url{https://arxiv.org/abs/2507.01353}, 2025.

\bibitem[Chow and Gessel(2007)]{chow_gessel}
C.-O. Chow and I.~M. Gessel.
\newblock On the descent numbers and major indices for the hyperoctahedral
  group.
\newblock \emph{Adv. in Appl. Math.}, 38\penalty0 (3):\penalty0 275--301, 2007.
\newblock \doi{10.1016/j.aam.2006.07.003}.
\newblock URL \url{https://doi.org/10.1016/j.aam.2006.07.003}.

\bibitem[David(1962)]{david_sample_mean_moderate_order}
H.~T. David.
\newblock The sample mean among the moderate order statistics.
\newblock \emph{Ann. Math. Statist.}, 33:\penalty0 1160--1166, 1962.
\newblock \doi{10.1214/aoms/1177704478}.
\newblock URL \url{https://doi.org/10.1214/aoms/1177704478}.

\bibitem[David(1963)]{david_sample_mean_among_extreme}
H.~T. David.
\newblock The sample mean among the extreme normal order statistics.
\newblock \emph{Ann. Math. Statist.}, 34:\penalty0 33--55, 1963.
\newblock \doi{10.1214/aoms/1177704241}.
\newblock URL \url{https://doi.org/10.1214/aoms/1177704241}.

\bibitem[Dey and Sivasubramanian(2022)]{dey_eulerian_CLT_Carlitz_for_types_B_D}
H.~K. Dey and S.~Sivasubramanian.
\newblock Eulerian central limit theorems and {C}arlitz identities in positive
  elements of classical {W}eyl groups.
\newblock \emph{J. Comb.}, 13\penalty0 (3):\penalty0 333--356, 2022.
\newblock \doi{10.4310/JOC.2022.v13.n3.a2}.
\newblock URL \url{https://doi.org/10.4310/JOC.2022.v13.n3.a2}.

\bibitem[{\relax DLMF}()]{NIST:DLMF}
{\relax DLMF}.
\newblock {\it NIST Digital Library of Mathematical Functions}.
\newblock \url{https://dlmf.nist.gov/}, Release 1.2.0 of 2024-03-15.
\newblock URL \url{https://dlmf.nist.gov/}.
\newblock F.~W.~J. Olver, A.~B. {Olde Daalhuis}, D.~W. Lozier, B.~I. Schneider,
  R.~F. Boisvert, C.~W. Clark, B.~R. Miller, B.~V. Saunders, H.~S. Cohl, and
  M.~A. McClain, eds.

\bibitem[Dmitrienko and Govindarajulu(1997)]{dmitrienko_demon}
A.~Dmitrienko and Z.~Govindarajulu.
\newblock On the ``demon'' problem of {Y}ouden.
\newblock \emph{Statist. Probab. Lett.}, 35\penalty0 (1):\penalty0 65--72,
  1997.
\newblock \doi{10.1016/S0167-7152(96)00217-9}.
\newblock URL \url{https://doi.org/10.1016/S0167-7152(96)00217-9}.

\bibitem[Dmitrienko and Govindarajulu(1998)]{dmitrienko_demon_exponential}
A.~Dmitrienko and Z.~Govindarajulu.
\newblock The ``demon'' problem of {Y}ouden: exponential case.
\newblock \emph{J. Appl. Statist.}, 25\penalty0 (4):\penalty0 517--523, 1998.
\newblock \doi{10.1080/02664769822981}.
\newblock URL \url{https://doi.org/10.1080/02664769822981}.

\bibitem[{Donoho} and {Tanner}(2010)]{donoho_tanner1}
D.~L. {Donoho} and J.~{Tanner}.
\newblock {Counting the faces of randomly-projected hypercubes and orthants,
  with applications.}
\newblock \emph{{Discrete Comput. Geom.}}, 43\penalty0 (3):\penalty0 522--541,
  2010.
\newblock \doi{10.1007/s00454-009-9221-z}.

\bibitem[Eckhoff(1979)]{eckhoff_radon_revisited}
J.~Eckhoff.
\newblock Radon's theorem revisited.
\newblock In \emph{Contributions to geometry ({P}roc. {G}eom. {S}ympos.,
  {S}iegen, 1978)}, pages 164--185. Birkh\"{a}user Verlag, Basel-Boston, Mass.,
  1979.

\bibitem[Eckhoff(1993)]{eckhoff_helly_radon_caratheodory}
J.~Eckhoff.
\newblock Helly, {R}adon, and {C}arath\'{e}odory type theorems.
\newblock In \emph{Handbook of convex geometry, {V}ol. {A}, {B}}, pages
  389--448. North-Holland, Amsterdam, 1993.

\bibitem[Ewald(1996)]{ewald_book_combinatorial_convexity_alg_geom}
G.~Ewald.
\newblock \emph{Combinatorial convexity and algebraic geometry}, volume 168 of
  \emph{Graduate Texts in Mathematics}.
\newblock Springer-Verlag, New York, 1996.
\newblock \doi{10.1007/978-1-4612-4044-0}.
\newblock URL \url{https://doi.org/10.1007/978-1-4612-4044-0}.

\bibitem[Frick et~al.(2025)Frick, Newman, and Pegden]{frick_newman_pegden}
F.~Frick, A.~Newman, and W.~Pegden.
\newblock Youden's demon is {S}ylvester's problem.
\newblock \emph{Mathematika}, 71\penalty0 (2):\penalty0 e70015, 2025.
\newblock URL \url{https://doi.org/10.1112/mtk.70015}.

\bibitem[Gawronski and Neuschel(2013)]{gawronski_neuschel_euler_frobenius}
W.~Gawronski and T.~Neuschel.
\newblock Euler-{F}robenius numbers.
\newblock \emph{Integral Transforms Spec. Funct.}, 24\penalty0 (10):\penalty0
  817--830, 2013.
\newblock \doi{10.1080/10652469.2012.762362}.
\newblock URL \url{https://doi.org/10.1080/10652469.2012.762362}.

\bibitem[Godland and Kabluchko(2022{\natexlab{a}})]{Schlafli_angle_sums}
T.~Godland and Z.~Kabluchko.
\newblock Angle sums of {S}chl\"afli orthoschemes.
\newblock \emph{Discrete Comput. Geom.}, 68\penalty0 (1):\penalty0 125--164,
  2022{\natexlab{a}}.
\newblock \doi{10.1007/s00454-021-00326-z}.
\newblock URL \url{https://doi.org/10.1007/s00454-021-00326-z}.

\bibitem[Godland and
  Kabluchko(2022{\natexlab{b}})]{godland_kabluchko_positive_hulls_rand_walks}
T.~Godland and Z.~Kabluchko.
\newblock Positive hulls of random walks and bridges.
\newblock \emph{Stochastic Process. Appl.}, 147:\penalty0 327--362,
  2022{\natexlab{b}}.
\newblock \doi{10.1016/j.spa.2022.01.019}.
\newblock URL \url{https://doi.org/10.1016/j.spa.2022.01.019}.

\bibitem[{Graham} et~al.(1994){Graham}, {Knuth}, and
  {Patashnik}]{graham_knuth_patashnik_book}
R.~L. {Graham}, D.~E. {Knuth}, and O.~{Patashnik}.
\newblock \emph{{Concrete mathematics: a foundation for computer science}}.
\newblock Amsterdam: Addison-Wesley Publishing Group, 2nd ed. edition, 1994.

\bibitem[Gr\"unbaum(2003)]{gruenbaum_book}
B.~Gr\"unbaum.
\newblock \emph{Convex {P}olytopes}, volume 221 of \emph{Graduate Texts in
  Mathematics}.
\newblock Springer-Verlag, New York, second edition, 2003.
\newblock \doi{10.1007/978-1-4613-0019-9}.
\newblock Prepared and with a preface by V. Kaibel, V. Klee and G. M. Ziegler.

\bibitem[Gusakova and Kabluchko(2025)]{gusakova_kabluchko_sylvester_beta}
A.~Gusakova and Z.~Kabluchko.
\newblock Sylvester's problem for beta-type distributions.
\newblock \emph{Statist. Probab. Letters}, 226:\penalty0 110482, 2025.
\newblock URL \url{https://doi.org/10.1016/j.spl.2025.110482}.

\bibitem[Henk et~al.(1997)Henk, Richter-Gebert, and
  Ziegler]{HenkRichterGebertZiegler1997BasicProperties}
M.~Henk, J.~Richter-Gebert, and G.~M. Ziegler.
\newblock Basic properties of convex polytopes.
\newblock In J.~E. Goodman and J.~O'Rourke, editors, \emph{Handbook of Discrete
  and Computational Geometry}, chapter~15, pages 243--270. CRC Press, Boca
  Raton, FL, 1997.

\bibitem[Hug(2013)]{hug_rev}
D.~Hug.
\newblock Random polytopes.
\newblock In \emph{Stochastic geometry, spatial statistics and random fields},
  volume 2068 of \emph{Lecture Notes in Math.}, pages 205--238. Springer,
  Heidelberg, 2013.
\newblock \doi{10.1007/978-3-642-33305-7_7}.
\newblock URL \url{https://doi.org/10.1007/978-3-642-33305-7_7}.

\bibitem[Hwang et~al.(2020)Hwang, Chern, and
  Duh]{hwang_etal_asympt_eulerian_recurrences}
H.-K. Hwang, H.-H. Chern, and G.-H. Duh.
\newblock An asymptotic distribution theory for {E}ulerian recurrences with
  applications.
\newblock \emph{Adv. in Appl. Math.}, 112:\penalty0 101960, 125, 2020.
\newblock \doi{10.1016/j.aam.2019.101960}.
\newblock URL \url{https://doi.org/10.1016/j.aam.2019.101960}.

\bibitem[Janson(2013)]{janson_euler_frobenius_rounding}
S.~Janson.
\newblock Euler-{F}robenius numbers and rounding.
\newblock \emph{Online J. Anal. Comb.}, 8:\penalty0 34, 2013.

\bibitem[Kabluchko(2020)]{kabluchko_poisson_zero}
Z.~Kabluchko.
\newblock Expected $f$-vector of the {P}oisson zero polytope and random convex
  hulls in the half-sphere.
\newblock \emph{Mathematika}, 66\penalty0 (4):\penalty0 1028--1053, 2020.
\newblock \doi{10.1112/mtk.12056}.
\newblock URL \url{https://doi.org/10.1112/mtk.12056}.

\bibitem[Kabluchko(2021{\natexlab{a}})]{kabluchko_angles_explicit_formula}
Z.~Kabluchko.
\newblock Angles of random simplices and face numbers of random polytopes.
\newblock \emph{Advances in Math.}, 380:\penalty0 107612, 2021{\natexlab{a}}.
\newblock \doi{https://doi.org/10.1016/j.aim.2021.107612}.
\newblock URL
  \url{https://www.sciencedirect.com/science/article/pii/S0001870821000505}.

\bibitem[Kabluchko(2021{\natexlab{b}})]{kabluchko_recursive_scheme}
Z.~Kabluchko.
\newblock Recursive scheme for angles of random simplices, and applications to
  random polytopes.
\newblock \emph{Discrete Comput. Geom.}, 66\penalty0 (3):\penalty0 902--937,
  2021{\natexlab{b}}.
\newblock \doi{10.1007/s00454-020-00259-z}.
\newblock URL \url{https://doi.org/10.1007/s00454-020-00259-z}.

\bibitem[Kabluchko and Klimovsky(2014)]{kabluchko_klimovsky_GREM}
Z.~Kabluchko and A.~Klimovsky.
\newblock Generalized {R}andom {E}nergy {M}odel at complex temperatures, 2014.
\newblock URL \url{https://arxiv.org/abs/1402.2142}.

\bibitem[Kabluchko and Marynych(2022)]{kabluchko_marynych_lah_distr}
Z.~Kabluchko and A.~Marynych.
\newblock Lah distribution: {S}tirling numbers, records on compositions, and
  convex hulls of high-dimensional random walks.
\newblock \emph{Probab. Theory Related Fields}, 184\penalty0 (3-4):\penalty0
  969--1028, 2022.
\newblock \doi{10.1007/s00440-022-01146-9}.
\newblock URL \url{https://doi.org/10.1007/s00440-022-01146-9}.

\bibitem[Kabluchko and Zaporozhets(2020)]{kabluchko_zaporozhets_absorption}
Z.~Kabluchko and D.~Zaporozhets.
\newblock Absorption probabilities for {G}aussian polytopes, and regular
  spherical simplices.
\newblock \emph{Adv. Appl. Probab.}, 52\penalty0 (2):\penalty0 588--616, 2020.
\newblock \doi{10.1017/apr.2020.7}.
\newblock URL \url{https://doi.org/10.1017/apr.2020.7}.

\bibitem[Kabluchko et~al.(2017)Kabluchko, Vysotsky, and Zaporozhets]{KVZ17}
Z.~Kabluchko, V.~Vysotsky, and D.~Zaporozhets.
\newblock Convex hulls of random walks: expected number of faces and face
  probabilities.
\newblock \emph{Adv. Math.}, 320:\penalty0 595--629, 2017.
\newblock \doi{10.1016/j.aim.2017.09.002}.
\newblock URL \url{https://doi.org/10.1016/j.aim.2017.09.002}.

\bibitem[Kabluchko et~al.(2019)Kabluchko, Marynych, Temesvari, and
  Th\"{a}le]{convex_hull_sphere}
Z.~Kabluchko, A.~Marynych, D.~Temesvari, and C.~Th\"{a}le.
\newblock Cones generated by random points on half-spheres and convex hulls of
  {P}oisson point processes.
\newblock \emph{Probab. Theory Related Fields}, 175\penalty0 (3-4):\penalty0
  1021--1061, 2019.
\newblock \doi{10.1007/s00440-019-00907-3}.
\newblock URL \url{https://doi.org/10.1007/s00440-019-00907-3}.

\bibitem[Kabluchko et~al.(2020)Kabluchko, Th{\"a}le, and
  Zaporozhets]{beta_polytopes}
Z.~Kabluchko, C.~Th{\"a}le, and D.~Zaporozhets.
\newblock Beta polytopes and {P}oisson polyhedra: {$f$}-vectors and angles.
\newblock \emph{Adv. Math.}, 374:\penalty0 107333, 2020.
\newblock \doi{https://doi.org/10.1016/j.aim.2020.107333}.
\newblock URL
  \url{http://www.sciencedirect.com/science/article/pii/S0001870820303613}.

\bibitem[Kabluchko et~al.(2026)Kabluchko, Steigenberger, and
  Th{\"a}le]{kabluchko_steigenberger_thaele_boob_beta_type}
Z.~Kabluchko, D.~A. Steigenberger, and C.~Th{\"a}le.
\newblock \emph{Random Simplices: {F}rom Beta-Type Distributions to
  High-Dimensional Volumes}, volume 2383 of \emph{Lecture Notes in
  Mathematics}.
\newblock Springer, 2026.

\bibitem[Kahle and Stump(2020)]{kahle_stump_counting_inversions_descents}
T.~Kahle and C.~Stump.
\newblock Counting inversions and descents of random elements in finite
  {C}oxeter groups.
\newblock \emph{Math. Comp.}, 89\penalty0 (321):\penalty0 437--464, 2020.
\newblock \doi{10.1090/mcom/3443}.
\newblock URL \url{https://doi.org/10.1090/mcom/3443}.

\bibitem[Kendall(1954)]{kendall_two_problems_youden}
M.~G. Kendall.
\newblock Two problems in sets of measurements.
\newblock \emph{Biometrika}, 41:\penalty0 560--564, 1954.
\newblock \doi{10.1093/biomet/41.3-4.560}.
\newblock URL \url{https://doi.org/10.1093/biomet/41.3-4.560}.

\bibitem[Kingman(1969)]{kingman_secants}
J.~F.~C. Kingman.
\newblock Random secants of a convex body.
\newblock \emph{J. Appl. Probability}, 6:\penalty0 660--672, 1969.
\newblock \doi{10.1017/s0021900200026693}.
\newblock URL \url{https://doi.org/10.1017/s0021900200026693}.

\bibitem[Kuchelmeister(2024)]{kuchelmeister_youdens_demon}
F.~Kuchelmeister.
\newblock On the probability of linear separability through intrinsic volumes.
\newblock Preprint at https://arxiv.org/abs/2404.12889, 2024.

\bibitem[MacMahon(1920)]{macmahon_divisors_numbers}
P.~A. MacMahon.
\newblock The {D}ivisors of {N}umbers.
\newblock \emph{Proc. London Math. Soc. (2)}, 19\penalty0 (4):\penalty0
  305--340, 1920.
\newblock \doi{10.1112/plms/s2-19.1.305}.
\newblock URL \url{https://doi.org/10.1112/plms/s2-19.1.305}.

\bibitem[Maehara and Martini(2017)]{maehara_martini_geometric_probab}
H.~Maehara and H.~Martini.
\newblock Geometric probability on the sphere.
\newblock \emph{Jahresber. Dtsch. Math.-Ver.}, 119\penalty0 (2):\penalty0
  93--132, 2017.
\newblock \doi{10.1365/s13291-017-0158-5}.
\newblock URL \url{https://doi.org/10.1365/s13291-017-0158-5}.

\bibitem[Maehara and Martini(2018)]{maehara_martini_sylvester_sphere}
H.~Maehara and H.~Martini.
\newblock An analogue of {S}ylvester's four-point problem on the sphere.
\newblock \emph{Acta Math. Hungar.}, 155\penalty0 (2):\penalty0 479--488, 2018.
\newblock \doi{10.1007/s10474-018-0814-y}.
\newblock URL \url{https://doi.org/10.1007/s10474-018-0814-y}.

\bibitem[Maehara and Martini(2024)]{maehara_martini_book}
H.~Maehara and H.~Martini.
\newblock \emph{Circles, Spheres and Spherical Geometry}.
\newblock Birkh\"auser Advanced Texts, Basler Lehrb\"ucher. Birkh\"auser Cham,
  2024.
\newblock \doi{10.1007/978-3-031-62776-7}.
\newblock URL \url{https://doi.org/10.1007/978-3-031-62776-7}.

\bibitem[Mez{\H{o}}(2020)]{mezo_book}
I.~Mez{\H{o}}.
\newblock \emph{Combinatorics and number theory of counting sequences}.
\newblock Discrete Mathematics and Its Applications. Boca Raton, CRC Press,
  2020.
\newblock \doi{10.1201/9781315122656}.

\bibitem[{Miles}(1971)]{miles}
R.E. {Miles}.
\newblock {Isotropic random simplices.}
\newblock \emph{{Adv. Appl. Probab.}}, 3:\penalty0 353--382, 1971.
\newblock \doi{10.2307/1426176}.

\bibitem[Morin(2025)]{morin_n_points_convex_position_regulag_k_gon}
L.~Morin.
\newblock Probability that $n$ points are in convex position in a regular
  $\kappa$-gon: {A}symptotic results.
\newblock \emph{Adv. Appl. Prob.}, pages 1--60, 2025.
\newblock URL \url{https://doi.org/10.1017/apr.2024.63}.

\bibitem[Nyul and R\'acz(2015)]{nyul:2015}
G.~Nyul and G.~R\'acz.
\newblock The $r$-{L}ah numbers.
\newblock \emph{Discrete Mathematics}, 338\penalty0 (10):\penalty0 1660--1666,
  2015.
\newblock URL \url{https://doi.org/10.1016/j.disc.2014.03.029}.

\bibitem[Panzo(2025)]{panzo_sylvester_random_walk}
H.~Panzo.
\newblock Sylvester's problem for random walks and bridges.
\newblock \emph{Statist. Probab. Lett.}, 219:\penalty0 Paper No. 110349, 5,
  2025.
\newblock ISSN 0167-7152,1879-2103.
\newblock \doi{10.1016/j.spl.2024.110349}.
\newblock URL \url{https://doi-org.ezp.slu.edu/10.1016/j.spl.2024.110349}.

\bibitem[Petersen(2015)]{petersen_book_eulerian_numbers}
T.~Kyle Petersen.
\newblock \emph{Eulerian numbers}.
\newblock Birkh\"{a}user Advanced Texts: Basel Textbooks.
  Birkh\"{a}user/Springer, New York, 2015.
\newblock \doi{10.1007/978-1-4939-3091-3}.
\newblock URL \url{https://doi.org/10.1007/978-1-4939-3091-3}.
\newblock With a foreword by Richard Stanley.

\bibitem[Reitzner(2010)]{reitzner_random_polys_survey}
M.~Reitzner.
\newblock Random polytopes.
\newblock In \emph{New perspectives in stochastic geometry}, pages 45--76.
  Oxford Univ. Press, Oxford, 2010.

\bibitem[Rogers(1961)]{rogers}
C.~A. Rogers.
\newblock An asymptotic expansion for certain {S}chl\"afli functions.
\newblock \emph{J. London Math. Soc.}, 36:\penalty0 78--80, 1961.
\newblock \doi{10.1112/jlms/s1-36.1.78}.

\bibitem[{Ruben} and {Miles}(1980)]{ruben_miles}
H.~{Ruben} and R.E. {Miles}.
\newblock {A canonical decomposition of the probability measure of sets of
  isotropic random points in $\mathbb R^n$.}
\newblock \emph{{J. Multivariate Anal.}}, 10:\penalty0 1--18, 1980.
\newblock \doi{10.1016/0047-259X(80)90077-9}.

\bibitem[Santal\'o(1976)]{lS76}
L.~Santal\'o.
\newblock \emph{Integral {G}eometry and {G}eometric {P}robability}, volume~1.
\newblock Addison-Wesley Publishing Company, Reading, 1976.

\bibitem[Schmidt and Simion(1997)]{schmidt_simion_geom_comb_eulerian_num}
F.~Schmidt and R.~Simion.
\newblock Some geometric probability problems involving the {E}ulerian numbers.
\newblock volume~4. 1997.
\newblock \doi{10.37236/1333}.
\newblock URL \url{https://doi.org/10.37236/1333}.
\newblock The Wilf Festschrift (Philadelphia, PA, 1996).

\bibitem[Schneider(1997)]{schneider_discrete_aspects_stoch_geom}
R.~Schneider.
\newblock Discrete aspects of stochastic geometry.
\newblock In \emph{Handbook of discrete and computational geometry}, CRC Press
  Ser. Discrete Math. Appl., pages 167--184. CRC, Boca Raton, FL, 1997.

\bibitem[Schneider and Weil(2008)]{schneider_weil_book}
R.~Schneider and W.~Weil.
\newblock \emph{Stochastic and {I}ntegral {G}eometry}.
\newblock Probability and its Applications. Springer--Verlag, Berlin, 2008.

\bibitem[Simion(1997)]{simion_convex_poly_enum}
R.~Simion.
\newblock Convex polytopes and enumeration.
\newblock \emph{Adv. in Appl. Math.}, 18\penalty0 (2):\penalty0 149--180, 1997.
\newblock \doi{10.1006/aama.1996.0505}.
\newblock URL \url{https://doi.org/10.1006/aama.1996.0505}.

\bibitem[Sloane~(editor)()]{sloane}
N.~J.~A. Sloane~(editor).
\newblock {The {O}n-{L}ine {E}ncyclopedia of {I}nteger {S}equences}.
\newblock https://oeis.org.

\bibitem[Solomon(1978)]{Solomon1975}
H.~Solomon.
\newblock \emph{Geometric probability}, volume No. 28 of \emph{Conference Board
  of the Mathematical Sciences Regional Conference Series in Applied
  Mathematics}.
\newblock Society for Industrial and Applied Mathematics, Philadelphia, PA,
  1978.
\newblock Ten lectures given at the University of Nevada, Las Vegas, Nev., June
  9--13, 1975.

\bibitem[Srivastava et~al.(2018)Srivastava, Boutiche, and
  Rahmani]{srivastava_eta_euler_frobenius}
H.~M. Srivastava, M.~A. Boutiche, and M.~Rahmani.
\newblock A class of {F}robenius-type {E}ulerian polynomials.
\newblock \emph{Rocky Mountain J. Math.}, 48\penalty0 (3):\penalty0 1003--1013,
  2018.
\newblock \doi{10.1216/RMJ-2018-48-3-1003}.
\newblock URL \url{https://doi.org/10.1216/RMJ-2018-48-3-1003}.

\bibitem[Vershik and
  Sporyshev(1992)]{vershik_sporyshev_asymptotic_faces_random_polyhedra1992}
A.~M. Vershik and P.~V. Sporyshev.
\newblock Asymptotic behavior of the number of faces of random polyhedra and
  the neighborliness problem.
\newblock \emph{Selecta Math. Soviet.}, 11\penalty0 (2):\penalty0 181--201,
  1992.
\newblock Selected translations.

\bibitem[Wendel(1962)]{wendel}
J.~G. Wendel.
\newblock A problem in geometric probability.
\newblock \emph{Math. Scand.}, 11:\penalty0 109--111, 1962.

\bibitem[Youden(1972)]{youden_sets_of_three_measurements}
W.~J. Youden.
\newblock Sets of three measurements.
\newblock \emph{Journal of Quality Technology}, 4\penalty0 (1):\penalty0
  40--44, 1972.
\newblock \doi{10.1080/00224065.1972.11980511}.
\newblock URL \url{https://doi.org/10.1080/00224065.1972.11980511}.

\end{thebibliography}

\end{document}